\documentclass{amsart}
\usepackage{fullpage,amsmath,mathtools,amsthm,amsfonts,amssymb,amscd,amsbsy,hyperref, amsbsy}
\usepackage{enumerate}
\usepackage{dsfont}
\usepackage{graphicx}

\newcommand{\id}{\operatorname{Id}}
\newcommand{\Vol}{\operatorname{Vol}}

\newcommand{\Iso}{\operatorname{Iso}}

\newcommand{\rot}{{\mathfrak r}}

\newcommand{\centr}{\operatorname{Z}}
\newcommand{\norm}{\operatorname{N}}

\newcommand{\aff}{\mathrm{Aff}}
\newcommand{\Flat}{\operatorname{\mathcal M_\text{\rm flat}}}
\newcommand{\teichflat}{\operatorname{\mathcal T_\text{\rm flat}}}

\newcommand{\N}{\mathds N}
\newcommand{\Z}{\mathds Z}
\newcommand{\R}{\mathds R}
\newcommand{\C}{\mathds C}
\renewcommand{\H}{\mathds H}

\newcommand{\GL}{\mathsf{GL}}
\newcommand{\SO}{\mathsf{SO}}
\renewcommand{\O}{\mathsf O}

\newcommand{\U}{\mathsf{U}}
\newcommand{\Sp}{\mathsf{Sp}}

\newcommand{\g}{\mathrm g}

\newcommand{\h}{\mathrm h}

\newcommand{\Diff}{\operatorname{Diff}}
\newcommand{\MCG}{\operatorname{MCG}}

\newcommand{\Out}{\operatorname{Out}}

\newcommand{\Inj}{\operatorname{Inj}\!}

\def\hs{\hskip.7pt}
\def\hh{\hskip.4pt}
\def\nh{\hskip-.7pt}
\def\nnh{\hskip-1.5pt}

\def\ve{\varepsilon}
\def\w{^{\phantom i}}

\allowdisplaybreaks

\usepackage{hyperref}
\hypersetup{
    pdftoolbar=true,
    pdfmenubar=true,
    pdffitwindow=false,
    pdfstartview={FitH},
    pdftitle={Teichmuller theory and collapse of flat manifolds},
    pdfauthor={Andrzej Derdzinski, Renato G. Bettiol and Paolo Piccione},
    pdfsubject={},
    pdfkeywords={},
    pdfnewwindow=true,
    colorlinks=true, 
    linkcolor=blue,
    citecolor=blue,
    urlcolor=blue,
}

\newtheorem{theorem}{Theorem}[]
\newtheorem{lemma}[theorem]{Lemma}
\newtheorem{proposition}[theorem]{Proposition}
\newtheorem{corollary}[theorem]{Corollary}
\newtheorem{definition}[theorem]{Definition}
\newtheorem*{bthm}{\sc Bieberbach Theorems}
\newtheorem{claim}[theorem]{Claim}

\newtheorem*{theorem*}{Theorem}
\newtheorem{mainthm}{\sc Theorem}

\newtheorem{maincor}[mainthm]{\sc Corollary}

\theoremstyle{remark}
\newtheorem{example}[theorem]{Example}
\newtheorem{remark}[theorem]{Remark}

\allowdisplaybreaks
\numberwithin{equation}{section}
\numberwithin{theorem}{section}

\title[Teichm\"uller theory and collapse of flat manifolds]{Teichm\"uller theory and collapse\\ of flat manifolds}
\author[R. G. Bettiol]{Renato G. Bettiol}
\author[A. Derdzinski]{Andrzej Derdzinski}
\author[P. Piccione]{Paolo Piccione}


\address{\begin{tabular}{lll}
University of Pennsylvania & & The Ohio State University\\
Department of Mathematics & & Department of Mathematics \\
209 South 33rd St  & & 231 W.~18th Avenue \\
Philadelphia, PA, 19104-6395, USA & & Columbus, OH 43210, USA\\
\emph{E-mail address}: {\tt rbettiol@math.upenn.edu} & & \emph{E-mail address}: {\tt andrzej@math.ohio-state.edu}
\\[.5cm]  Universidade de S\~ao Paulo\\ Departamento de Matem\'atica\\ Rua do Mat\~ao 1010\\
S\~ao Paulo, SP 05508--090, Brazil\\ \emph{E-mail address}: {\tt piccione@ime.usp.br}
\end{tabular}
}
\date{\today}

\begin{document}
\begin{abstract}
We provide an algebraic description of the Teichm\"uller space and moduli space of flat metrics on a closed manifold or orbifold and study its boundary, which consists of (isometry classes of) flat orbifolds to which the original object may collapse. It is also shown that every closed flat orbifold can be obtained by collapsing closed flat manifolds, and the collapsed limits of closed flat $3$-manifolds are classified.
\end{abstract}
\maketitle


\section{Introduction}
A fundamental question in Riemannian geometry is whether there exist deformations of a given manifold that preserve certain curvature conditions and, if so, what is the nature of the limiting spaces. In this paper, we study how flat manifolds and flat orbifolds can be deformed while keeping them flat; and, in particular, how they collapse and what are the possible limits. 

The Gromov-Hausdorff limit of a sequence of closed flat manifolds is clearly a flat Alexandrov space, that is, it has curvature bounded both from above and from below by zero, in triangle comparison sense. Our first result implies that the only singularities that may arise in such collapsed flat limits are the mildest possible, and that any flat space with singularities of this type admits a smooth flat resolution:

\begin{mainthm}\label{mainthm:limits}
The Gromov-Hausdorff limit of a sequence of closed flat manifolds is a closed flat orbifold. Conversely, every closed flat orbifold is the Gromov-Hausdorff limit of a sequence of closed flat manifolds.
\end{mainthm}

The formation of orbifold singularities in the collapse of smooth flat manifolds can be easily seen already in dimension $2$. Consider flat Klein bottles as rectangles with the usual boundary identifications. The Gromov-Hausdorff limit obtained by shrinking the lengths of a pair of opposite sides is either $S^1$ or a closed interval (a flat $1$-orbifold), depending on whether the identification of these sides preserves or reverses orientation.

A simple diagonal argument, combined with Theorem~\ref{mainthm:limits}, implies that the collection of closed flat orbifolds is closed in the Gromov-Hausdorff topology. In light of the fact that every object in this collection is the limit of smooth flat manifolds, it would be interesting to determine whether every orbifold with $\sec\geq0$ is the limit of manifolds with $\sec\geq0$; see Remark~\ref{rem:cheegerdef}. An important and currently open question is whether every finite-dimensional Alexandrov space with $\mathrm{curv}\geq K$ is the limit of smooth manifolds with $\sec\geq K$. In this context, recall that Alexandrov spaces of dimension $3$ and $4$ are homeomorphic to orbifolds~\cite{fernando-luis,harvey-searle}.

Analyzing deformations and limits of flat orbifolds leads to investigating the moduli space $\Flat(\mathcal O)$ of flat metrics on a fixed flat orbifold~$\mathcal O$ and its ideal boundary. The nature of this moduli space is very reminiscent of the classical Teichm\"uller theory for hyperbolic surfaces, in that $\Flat(\mathcal O)$ is the quotient of a \emph{Teichm\"uller space} $\teichflat(\mathcal O)$, diffeomorphic to an open ball, under the action of a (discrete) mapping class group. This fits the picture of a deformation theory for geometric structures of much larger scope pioneered by Thurston~\cite{thurston} and further developed in \cite{Baues2000,Goldman88}; see Subsection~\ref{sub:defXG}. Around 45 years ago, Wolf~\cite{Wolf73} identified the moduli space $\Flat(M)$ of flat metrics on a flat manifold $M$. However, even in this special case, a systematic and unified treatment of the Teichm\"uller theory of flat manifolds does not seem to be available in the literature, despite some scattered results, e.g.~\cite{Kang06,KangKim03}.
Towards this goal, we establish the following algebraic description of the Teichm\"uller space of flat metrics on a flat orbifold, which provides a straightforward method to compute it:

\begin{mainthm}\label{mainthm:teichdescription}
Let $\mathcal O$ be a closed flat orbifold, and denote by $W_i$, $1\leq i\leq l$, the isotypic components of the orthogonal representation of its holonomy group. Each $W_i$ consists of $m_i$ copies of the same irreducible representation, and we write $\mathds K_i$ for $\R$, $\C$, or $\H$, according to this irreducible being of real, complex, or quaternionic type.
The Teichm\"uller space $\teichflat(\mathcal O)$ is diffeomorphic to:
\begin{equation*}
\teichflat(\mathcal O)\cong\prod_{i=1}^l\frac{\GL(m_i,\mathds K_i)}{\O(m_i,\mathds K_i)},
\end{equation*}
where $\GL(m,\mathds K)$ is the group of $\mathds K$-linear automorphisms of $\mathds K^m$ and $\O(m,\mathds K)$ stands for $\O(m)$, $\U(m)$, or $\Sp(m)$, when $\mathds K$ is, respectively, $\R$, $\C$, or $\H$.
In particular, $\teichflat(\mathcal O)$ is real-analytic and diffeomorphic to~$\R^d$.
\end{mainthm}

The dimension $d=\dim\teichflat(\mathcal O)$ is easily computed as the sum of the dimensions $d_i\geq1$ of the factors $\O(m_i,\mathds K_i)\backslash\GL(m_i,\mathds K_i)\cong\R^{d_i}$, $1\leq i\leq l$, which are given by:
\begin{equation*}
d_i=\begin{cases}\frac12m_i(m_i+1),&\text{ if }\mathds K_i=\R,\\
m_i^2,&\text{ if }\mathds K_i=\C,\\
m_i(2m_i-1),&\text{ if }\mathds K_i=\H.
\end{cases}
\end{equation*}
In particular, since the holonomy representation of a flat manifold is reducible~\cite{hiss-szczepa}, see Theorem~\ref{thm:redholonomy}, it follows that $l\geq2$, and hence $d\geq2$. This implies the following:

\begin{maincor}
Every flat manifold admits nonhomothetic flat deformations.
\end{maincor}

The situation is different for flat orbifolds, which can be \emph{rigid}. Examples of orbifolds with irreducible holonomy representation, i.e., $l=1$, which consequently admit no nonhomothetic flat deformations, already appear in dimension $2$: for instance, flat equilateral triangles; see Subsection~\ref{subsec:flat2orbifolds} for more examples.

Since flat orbifolds are locally isometric to Euclidean spaces, the most interesting aspects of their geometry are clearly global. Thus, it is not surprising that issues related to holonomy play a central role in developing this Teichm\"uller theory. As an elementary case illustrating Theorem~\ref{mainthm:teichdescription}, consider the complete absence of holonomy: flat $n$-dimensional tori $T^n$ can be realized as parallelepipeds spanned by linearly independent vectors $v_1,\dots, v_n\in\R^n$, with opposite faces identified. Flat metrics on $T^n$ correspond to different choices of $v_1,\dots, v_n$, up to ambiguities arising from rigid motions in $\R^n$, or relabelings and subdivisions of the parallelepiped into smaller pieces with boundary identifications. More precisely, it is not difficult to see that $\Flat(T^n)=\O(n)\backslash\GL(n,\R)/\GL(n,\Z)$. In this case, $\teichflat(T^n)=\O(n)\backslash\GL(n,\R)\cong\R^{n(n+1)/2}$ is the space of inner products on $\R^n$, and $\Flat(T^n)=\teichflat(T^n)/\GL(n,\Z)$; see Subsection~\ref{subsec:tori} for details.

Isometry classes of collapsed limits of $T^n$ correspond to points in the ideal boundary of $\Flat(T^n)$. A~more tangible object is the ideal boundary of $\teichflat(T^n)$, formed by positive-semidefinite $n\times n$ matrices and stratified by their rank $k$, with $0\leq k< n$, which in a sense correspond to the $k$-dimensional flat tori $T^k$ to which $T^n$ can collapse. Nevertheless, we warn the reader that the Gromov-Hausdorff distance \emph{does not} extend continuously to this boundary. For instance, collapsing the $2$-dimensional square torus along a line of slope $p/q$, with $p,q\in\Z$, $\gcd(p,q)=1$, produces as Gromov-Hausdorff limit the circle $S^1$ of length $(p^2+q^2)^{-1/2}$, while collapsing it along any nearby line with irrational slope produces as limit a single point.

Recall that there are precisely 17 affine classes of flat $2$-orbifolds, corresponding to the 17 \emph{wallpaper groups}, see for instance \cite{davis-notes}. The underlying topological space of $\mathcal O$ is a $2$-manifold $|\mathcal O|$, possibly with boundary; namely, the disk ~$D^2$, sphere $S^2$, real projective plane $\R P^2$, torus $T^2$, Klein bottle $K^2$, cylinder $S^1\times I$, or M\"obius band $M^2$. The singularities that occur in the interior are \emph{cone points}, labeled with a positive integer $n\in\N$, specifying that the local group is the cyclic group $\Z_n\subset\SO(2)$. Singularities that occur on the boundary are \emph{corner reflectors}, labeled by a positive integer $m\in\N$, specifying that the local group is the dihedral group $D_m\subset\O(2)$ of $2m$ elements. Following Davis~\cite{davis-notes}, if a $2$-orbifold $\mathcal O$ has $\ell$ cone points labeled $n_1,\ldots,n_\ell$, and $k$ corner reflectors labeled $m_1,\ldots,m_k$, then it is denoted $|\mathcal O|(n_1,\dots,n_\ell;m_1,\ldots,m_k)$. 

By direct inspection, we verify that only 10 out of the 17 flat orbifolds of dimension 2 (see Table~\ref{table:flat2orbifolds}) arise as Gromov-Hausdorff limits of closed flat 3-manifolds:

\begin{mainthm}\label{mainthm:2orbifolds}
The Gromov-Hausdorff limit of a sequence of closed flat $3$-manifolds is either a closed flat $3$-manifold, or one of the following collapsing cases: point, closed interval, circle, $2$-torus, Klein bottle, M\"obius band, cylinder, disk with singularities $D^2(4;2)$, $D^2(3;3)$, or $D^2(2,2;)$, sphere with singularities $S^2(3,3,3;)$ or $S^2(2,2,2,2;)$, and the real projective plane with singularities $\R P^2(2,2;)$.
\end{mainthm}

The remaining 7 flat $2$-orbifolds must arise as Gromov-Hausdorff limits of flat manifolds of dimensions~$\geq4$. For example, it is easy to see that a flat rectangle $D^2(;2,2,2,2)$ can be realized as the collapsed limit of a product $K^2\times K^2$ of Klein bottles. However, there does not seem to be a readily available method to decide what is the \emph{lowest dimension} $N(\mathcal O)$ of a flat manifold that collapses to a given flat orbifold $\mathcal O$. Heuristically, we expect that an upper bound for $N(\mathcal O)$ could be derived from the structure of the singular set of $\mathcal O$. For~instance, adding one dimension for each cone point, one could expect to replace the rotation that fixes the cone point with a ``screw motion" having a translation component in this new direction, which hence acts freely and \emph{resolves} the cone singularity.

Another method to estimate $N(\mathcal O)$, following the proof of Theorem~\ref{mainthm:limits}, is to estimate the lowest dimension of a flat manifold with prescribed holonomy group $H$ isomorphic to the holonomy group of $\mathcal O$. Such a flat manifold always exists by a result of Auslander and Kuranishi~\cite{auslanderkuranishi57}, see Theorem~\ref{thm:AKthm}, and determining its lowest dimension is a well-known open question, see Szczepa{\'n}ski~\cite[Problem~1]{sczcproblems}. An answer to this problem, and consequently an upper estimate for $N(\mathcal O)$, is available when the group $H$ is cyclic, an elementary abelian $p$-group, dihedral, semidihedral, a generalized quaternion group, or a simple group $\mathsf{PSL}(2,p)$ with $p$ prime. 
According to Szczepa{\'n}ski~\cite{sczcproblems}, the difficulty in establishing more comprehensive results in this direction is related to the difficulty in computing the second group cohomology of $H$ with special coefficients.

There are several other questions related to the results in this paper, two of which we would like to emphasize. The first is to characterize algebraically which isotypic components of the holonomy group of a flat manifold produce as Gromov-Hausdorff limit another flat manifold (instead of a singular flat orbifold) when collapsed. The second is to what extent our results generalize to the class of \emph{almost flat manifolds}.

This paper is organized as follows. Preliminary facts about closed flat manifolds and flat orbifolds are recalled in Section~\ref{sec:prelims}, including the Bieberbach Theorems and the classification of flat $2$-orbifolds. Limits of closed flat manifolds are studied in Section~\ref{sec:sequences}, which contains the proof of Theorem~\ref{mainthm:limits}. Section~\ref{sec:algdescr} deals with Teichm\"uller spaces and moduli spaces of flat metrics, establishing the algebraic characterization given by Theorem~\ref{mainthm:teichdescription}.
Several examples of Teichm\"uller spaces that can be computed by applying Theorem~\ref{mainthm:teichdescription} are discussed in Section~\ref{sec:exa}. Section~\ref{sec:examples} contains the classification of limits of flat $3$-manifolds and the proof of Theorem~\ref{mainthm:2orbifolds}. Finally, Appendix~\ref{sec:appflat2orbs} contains details on how to recognize flat $2$-orbifolds given as quotients of $\R^2$ by crystallographic groups, which are used in Section~\ref{sec:examples}, and Appendix~\ref{sec:appendixB} gives an alternative elementary proof that Gromov-Hausdorff limits of flat tori are (lower dimensional) flat tori.

\subsection*{Acknowledgements}
It is a pleasure to thank Alexander Lytchak, John Harvey, Karsten Grove, and Curtis Pro for suggestions regarding equivariant Gromov-Hausdorff convergence, Burkhard Wilking for comments on realizing flat orbifolds as limits of flat manifolds, and Andrzej Szczepa{\'n}ski for conversations about the holonomy representation of flat manifolds. Part of this work was done during a visit of the first named author to the Max Planck Institute for Mathematics in Bonn, Germany, and of the third named author to the University of Notre Dame, USA; they would like to thank these institutions for providing excellent working conditions.
The second and third named authors are partially supported by a FAPESP-OSU 2015 Regular Research Award (FAPESP grant: 2015/50315-3).

\section{Flat manifolds and flat orbifolds}\label{sec:prelims}

In this section, we recall basic facts about closed flat manifolds and flat orbifolds.

\subsection{Orbifolds}
A \emph{Riemannian orbifold} $\mathcal O$ is a metric space which is locally isometric to orbit spaces of isometric actions of finite groups on Riemannian manifolds. Geometric properties, such as curvature, may be defined via these local isometries. Following Thurston~\cite{thurston}, $\mathcal O$ is called \emph{good} if it is globally isometric to such an orbit space. Any orbifold $\mathcal O$ has a \emph{universal orbifold covering} $\widetilde{\mathcal O}$, with a discrete isometric action by deck transformations of its \emph{orbifold fundamental group} $\pi_1^{orb}(\mathcal O)$. 
The orbifold $\mathcal O$ has non-empty boundary (as an Alexandrov space) if and only if $\pi_1^{orb}(\mathcal O)$ contains a reflection, that is, an involution with fixed point set of codimension $1$.

Every Riemannian orbifold $\mathcal O$ of dimension $n$ has a frame bundle $\operatorname{Fr}(\mathcal O)$, which is a (smooth) Riemannian manifold with an almost free isometric $\O(n)$-action whose orbit space is $\operatorname{Fr}(\mathcal O)/\O(n)\cong\mathcal O$. In particular, it follows that every Riemannian orbifold is the (continuous) Gromov-Hausdorff limit of Riemannian manifolds $\{(\operatorname{Fr}(\mathcal O),\g_t)\}_{t\geq0}$, where $\g_t$ is the Cheeger deformation of some invariant metric with respect to the $\O(n)$-action, see \cite[\S 6.1]{mybook}. 
For details on the basic geometry and topology of orbifolds, see \cite{bh-book,davis-notes,kleiner-lott,thurston}.

\begin{remark}\label{rem:cheegerdef}
At first sight, since Cheeger deformations preserve $\sec\geq0$, the above facts may seem to provide an approach to solving the question in the Introduction about realizing an orbifold $\mathcal O$ with $\sec\geq0$ as a limit of manifolds with $\sec\geq0$. However, it is in general difficult to endow $\mathrm{Fr}(\mathcal O)$ with $\sec\geq0$. This problem is related to the well-known \emph{converse question} to the Soul Theorem of Gromoll and Meyer, of which vector bundles over closed manifolds with $\sec\geq0$ admit metrics with $\sec\geq0$. 

In the very special case of a spherical orbifold $\mathcal O=S^n/\Gamma$, such as the spherical suspension of $\R P^2$, the frame bundle $\mathrm{Fr}(\mathcal O)=\O(n+1)/\Gamma$ clearly admits $\sec\geq0$. Thus, Cheeger deformations allow to approximate~$\mathcal O$ by manifolds with $\sec\geq0$. However, it is unclear whether that can be done keeping the \emph{same optimal lower curvature bound} $\sec\geq1$ of the limit $\mathcal O$.
\end{remark}

\subsection{Bieberbach Theorems}\label{sub:BiebGroups}
A discrete group $\pi$ of isometries of $\R^n$ is called \emph{crystallographic} if it has compact fundamental domain in $\R^n$, so that $\mathcal O=\R^n/\pi$ is a closed flat orbifold. A torsion-free crystallographic group $\pi$ is called a \emph{Bieberbach group}, and in this case the action of $\pi$ on $\R^n$ is free, so $M=\R^n/\pi$ is a closed flat manifold. 
Conversely, by the Killing-Hopf Theorem, it is well-known that if a closed manifold $M$ of dimension $n\geq2$ carries a flat Riemannian metric, then its universal covering is $\R^n$ and its fundamental group is isomorphic to a Bieberbach group. Similarly, by a result of Thurston~\cite{thurston} (see~\cite{good-orbifolds}), if a closed Riemannian orbifold $\mathcal O$ of dimension $n\geq2$ is flat, then it is good, its universal orbifold covering is $\R^n$ and its orbifold fundamental group is isomorphic to a crystallographic group.
In all of these cases, we denote by $\g_\pi$ the flat metric for which the quotient map $(\R^n,\g_\text{flat})\to(\R^n/\pi,\g_\pi)$ is a Riemannian covering.

Denote by $\aff(\R^n):=\GL(n)\ltimes\R^n$ the group of affine transformations of $\R^n$, and by $\Iso(\R^n):=\O(n)\ltimes\R^n$ be the subgroup of rigid motions, that is, isometries of the Euclidean space $(\R^n,\g_\text{flat})$.
We write elements of $\aff(\R^n)$ and $\Iso(\R^n)$ as pairs $(A,v)$, with $A\in\GL(n)$ or $\O(n)$, and $v\in\R^n$. The group operation is given by
\begin{equation}\label{eq:gpop}
(A,v)\cdot (B,w)=(AB, Aw+v),
\end{equation}
and clearly $(A,v)^{-1}=(A^{-1},-A^{-1}v)$. The natural action of these groups on $\R^n$ is given by $(A,v)\cdot w=Aw+v$.

Consider the projection homomorphism:
\[\rot:\aff(\R^n)\longrightarrow\GL(n), \quad \rot(A,v)=A.\]
Given a crystallographic group $\pi\subset\Iso(\R^n)$, its image $H_\pi:=\rot(\pi)$ is called the \emph{holonomy} of $\pi$.
Let $L_\pi$ denote the kernel of the restriction of $\rot$ to $\pi$, so that we have a short exact sequence
\begin{equation}\label{eq:exsecBieber}
1\longrightarrow L_\pi\longrightarrow\pi\longrightarrow H_\pi\longrightarrow 1.
\end{equation}
The group $L_\pi$, which consists of elements of the form $(\id,v)$, with $v\in\R^n$, is the maximal normal abelian subgroup of $\pi$ and is naturally identified with a subgroup of $\R^n$.

The \emph{Bieberbach Theorems}~\cite{bieberbach,bieberbach2}, see also \cite{buser-bieber,charlap,szczepa-book}, provide the essential facts about the groups in \eqref{eq:exsecBieber} and have equivalent algebraic and geometric formulations:

\begin{bthm}[\sc Algebraic version]
The following hold:
\begin{enumerate}[\rm I.]
\item If $\pi\subset\Iso(\R^n)$ is a crystallographic group, then $H_\pi$ is finite and $L_\pi$ is a lattice that spans $\R^n$.
\item Let $\pi,\pi'\subset\Iso(\R^n)$ be crystallographic subgroups. If there exists an isomorphism $\phi\colon\pi\to\pi'$, then $\phi$ is a conjugation in $\aff(\R^n)$, i.e., there exists $(A,v)\in\aff(\R^n)$ such that $\phi(B,w)=(A,v)\cdot (B,w)\cdot (A,v)^{-1}$ for all $(B,w)\in\pi$.
\item For all $n$, there are only finitely many isomorphism classes of crystallographic subgroups of $\Iso(\R^n)$.
\end{enumerate}
\end{bthm}

\begin{bthm}[\sc Geometric version]
The following hold:
\begin{enumerate}[\rm I.]
\item If $(\mathcal O,\g)$ is a closed flat orbifold with $\dim \mathcal O=n$, then $(\mathcal O,\g)$ is covered by a flat torus of dimension $n$, and the covering map is a local isometry.
\item If $\mathcal O$ and $\mathcal O'$ are closed flat orbifolds of the same dimension with isomorphic fundamental groups, then $\mathcal O$ and $\mathcal O'$ are affinely equivalent.
\item For all $n$, there are only finitely many affine equivalence classes of closed flat orbifolds of dimension $n$.
\end{enumerate}
\end{bthm}

The list of (affine equivalence classes of) closed flat orbifolds of dimension $n$, which, by the above, is in bijective correspondence with the list of (affine conjugate classes of) crystallographic groups in $\Iso(\R^n)$, is known for some small values of $n$:
\begin{itemize}
\item If $n=2$, there are 17 examples, corresponding to the 17 wallpaper groups. They are listed in Table~\ref{table:flat2orbifolds}, using the notation for cone points and corner reflectors as in the Introduction (following Davis~\cite{davis-notes}). The corresponding wallpaper groups $\pi$ are identified by their crystallographic notation, followed by Conway's notation~\cite{conway} in parenthesis, and their holonomy group $H_\pi$ is also indicated;
\begin{table}[ht]
 \begin{center}
  \caption{Flat $2$-dimensional orbifolds}\label{table:flat2orbifolds}
  \begin{tabular}{rlllll}
    \hline\noalign{\smallskip}
$\#$   & Orbifold $\R^2/\pi$ & Topology & Geometry & $\pi$ & $H_\pi$
      \\ \noalign{\smallskip}\hline
      \noalign{\smallskip}
\noalign{\smallskip}
(1) & $D^2(;3,3,3)$ & $D^2$ & equilateral triangle& p3m1 ($^*333$) & $D_3$\\
\noalign{\smallskip}
(2) & $D^2(;2,3,6)$ & $D^2$ & triangle with angles $\frac{\pi}{2}$, $\frac{\pi}{3}$, $\frac{\pi}{6}$& p6m ($^*632$) & $D_6$\\
\noalign{\smallskip}
(3)& $D^2(;2,4,4)$ & $D^2$ & triangle with angles $\frac{\pi}{2}$, $\frac{\pi}{4}$, $\frac{\pi}{4}$& p4m ($^*442$) & $D_4$ \\
\noalign{\smallskip}
(4)& $D^2(;2,2,2,2)$ & $D^2$ & rectangle& pmm ($^*2222$) & $D_2$ \\
\noalign{\smallskip}
(5)  & $D^2(2;2,2)$ & $D^2$ &  \begin{tabular}{@{}l@{}}quotient of a square by group $\Z_2$ generated \\ by the rotation of $\pi$ around its center \end{tabular}& cmm ($2^*22$) & $D_2$\\
\noalign{\smallskip}
(6)  & $D^2(4;2)$ & $D^2$ & \begin{tabular}{@{}l@{}} quotient of a square by group $\Z_4$ generated \\ by the rotation of $\frac\pi2$ around its center \end{tabular}& p4g ($4^*2$) & $D_4$\\
\noalign{\smallskip}
(7) & $D^2(3;3)$ & $D^2$ &  \begin{tabular}{@{}l@{}}quotient of an equilateral triangle by group $\Z_3$ \\  generated by the rotation of $\frac\pi3$ around its center \end{tabular}& p31m ($3^*3$) & $D_3$ \\
\noalign{\smallskip}
(8) & $D^2(2,2;)$ & $D^2$ & \begin{tabular}{@{}l@{}} \emph{half pillowcase}: quotient of $S^2(2,2,2,2;)$ by \\ reflection about the equator \end{tabular}& pmg ($22^*$) & $D_2$ \\
\noalign{\smallskip}
(9) & $S^2(2,2,2,2;)$ & $S^2$ & \emph{pillowcase}: Alexandrov double of rectangle & p2 ($2222$) & $D_2$ \\
\noalign{\smallskip}
(10)  & $S^2(3,3,3;)$ & $S^2$ & \emph{$333$-turnover}: Alexandrov double of $D^2(;3,3,3)$& p3 ($333$) & $\Z_3$ \\
\noalign{\smallskip}
(11) & $S^2(2,3,6;)$ & $S^2$ & \emph{$236$-turnover}: Alexandrov double of $D^2(;2,3,6)$& p6 ($632$) & $\Z_6$ \\
\noalign{\smallskip}
(12) & $S^2(2,4,4;)$ & $S^2$ & \emph{$244$-turnover}: Alexandrov double of $D^2(;2,4,4)$ & p4 ($442$) & $D_4$ \\
\noalign{\smallskip}
(13) & $\R P^2(2,2;)$ & $\R P^2$ & quotient of $S^2(2,2,2,2;)$ by antipodal map & pgg ($22\times$)  & $D_2$\\
\noalign{\smallskip}
(14)  & $T^2$ & $T^2$ & $2$-torus& p1 ($\circ$) & $\{1\}$\\
\noalign{\smallskip}
(15)  & $K^2$ & $K^2$ & Klein bottle & pg ($\times\times$) & $\Z_2$ \\
\noalign{\smallskip}
(16) & $S^1\times I$ & $S^1\times I$ & Cylinder & pm (${^*}{^*}$) & $\Z_2$ \\
\noalign{\smallskip}
(17) & $M^2$ & $M^2$ & M\"obius band & cm ($^*\times$) & $\Z_2$  \\ \noalign{\smallskip}\hline
\end{tabular}
\end{center}
\end{table}

\item If $n=3$, there are 219 examples, corresponding to the 219 space groups classified (independently) by Barlow, Fedorov, and Sch\"onflies, in the 1890s;
\item If $n=4$, there are 4,783 examples classified by Brown, B\"ulow, Neub\"user, Wondratschek, and Zassenhaus~\cite{crystalbook};
\item If $n=5$ and $n=6$, there are respectively 222,018 and 28,927,922 examples, obtained with the computer program CARAT, see Plesken and Schulz~\cite{plesken-schulz}.
\end{itemize}
The sublist of (affine equivalence classes of) closed flat manifolds of dimension $n$ is also known for the above values of $n$, and is considerably shorter:
\begin{itemize}
\item If $n=2$, there are 2 examples: the torus $T^2$ and the Klein bottle $K^2$;
\item If $n=3$, there are 10 examples, obtained by Hantzsche and Wendt~\cite{hantwend}, see Wolf~\cite[Thm.~3.5.5, 3.5.9]{Wolfbook};
\item If $n=4$, there are 74 examples, obtained by Calabi, see Wolf~\cite[Sec.\ 3.6]{Wolfbook};
\item If $n=5$ and $n=6$, there are respectively 1,060 and 38,746 examples, obtained with the computer program CARAT, see Cid and Schulz~\cite{cids}.
\end{itemize}

\subsection{Holonomy group}
Let $\pi\subset\Iso(\R^n)$ be a crystallographic group and consider its holonomy group $H_\pi$.

\begin{lemma}\label{lem:lpiinv}
The lattice $L_\pi\subset\R^n$ is invariant under the orthogonal action of $H_\pi$.
\end{lemma}

\begin{proof}
This is an easy consequence of normality of $L_\pi$ in $\pi$. Namely, given $w\in L_\pi$ and $A\in H_\pi$, let $v\in\R^n$ be such that $(A,v)\in\pi$. Then, by normality, there exists $w'\in L_p$ such that $(A,Aw+v)=(A,v)\cdot(\id, w)=(\id, w')\cdot(A,v)=(A,w'+v)$, i.e., $Aw=w'$. Thus, $H_\pi(L_\pi)\subset L_\pi$.
\end{proof}

If $\pi\subset\Iso(\R^n)$ is a Bieberbach group, then the orthogonal representation of its holonomy $H_\pi$ on $\R^n$ is identified with the holonomy representation of the flat manifold $M=\R^n/\pi$.
In particular, notice that any two flat metrics on $M$ have isomorphic holonomy groups. Furthermore, the Betti numbers of $M$ are given by $b_k(M)=\dim(\wedge^k\R^n)^{H_\pi}$, that is, the dimension of the subspace of $\wedge^k\R^n$ fixed by the induced orthogonal representation of $H_\pi$.

The closed flat manifold $M=\R^n/\pi$ can also be seen as the orbit space of a free isometric action of $H_\pi$ on the flat torus $\R^n/L_\pi$. Namely, since $L_\pi$ is normal in $\pi$, the projection map $\R^n/L_\pi\to M$ is a regular (Riemannian) covering whose group of deck transformations is identified with $\pi/L_\pi\cong H_\pi$. In particular, it follows that $\Vol(M,\g_\pi)=\vert H_\pi\vert^{-1}\,\Vol(\R^n/L_\pi)$. 

The free action of $H_\pi$ on $\R^n/L_\pi$ with orbit space $M$ can be described explicitly.
First, note that since $L_\pi$ is $H_\pi$-invariant (Lemma~\ref{lem:lpiinv}), the natural action of $H_\pi$ on $\R^n$ descends to an action of $H_\pi$ on the torus $\R^n/L_\pi$. Given $A\in H_\pi$, denote by $\overline A\colon \R^n/L_\pi\to\R^n/L_\pi$ the induced map, and let $v\in\R^n$ be such that $(A,v)\in \pi$. Such $v$ is unique, up to translations in $L_\pi$, for if $v,v'\in\R^n$ are such that $(A,v),(A,v')\in\pi$, then $(A,v)\cdot(A,v')^{-1}=(\id, v-v')\in L_\pi$, i.e., $v-v'\in L_\pi$. We thus have a map $H_\pi\ni A\mapsto v_A\in\R^n/L_\pi$, defined by $v_A:=v+L_\pi$, where $(A,v)\in\pi$. Clearly, $v_{A\cdot A'}=\overline A(v_{A'})+v_A$ for all $A,A'\in H_\pi$, where $+$ is the group operation on $\R^n/L_\pi$.
With this notation, the free action of $H_\pi$ on $\R^n/L_\pi$ is:
\[\phantom{,\qquad A\in H_\pi,\,\overline x\in\R^n/L_\pi.}A\cdot \overline x:=\overline A(\overline x)+v_A,\qquad A\in H_\pi,\, \overline x\in\R^n/L_\pi.\]

We conclude this with two very useful results about holonomy groups $H_\pi$. First, by a celebrated theorem of Auslander and Kuranishi~\cite{auslanderkuranishi57}, see also Wolf~\cite{Wolfbook}, there are no obstructions on $H_\pi$; more precisely:

\begin{theorem}\label{thm:AKthm}
Any finite group is the holonomy group $H_\pi$ of a closed flat manifold.
\end{theorem}

Second, Hiss and Szczepa\'nski~\cite{hiss-szczepa} established the following remarkable result:

\begin{theorem}\label{thm:redholonomy} 
For any Bieberbach group $\pi\subset\Iso(\R^n)$, the orthogonal action of~$H_\pi$ on $\R^n$ is reducible.
\end{theorem}

\begin{remark}
The proof of Theorem~\ref{thm:redholonomy} is rather involved, however it is worth observing that it becomes elementary if $H_\pi$ has nontrivial center. Namely, since the action of $\pi$ on $\R^n$ is free, given $(A,v)\in\pi$, $(A,v)\neq(\mathrm{Id},0)$, there does not exist $x\in\R^n$ such that $Ax+v=x$, i.e., $(A-\mathrm{Id})x=-v$. Thus, $(A-\mathrm{Id})$ is not invertible, hence $1$ is an eigenvalue of $A$. This gives a nontrivial orthogonal decomposition $\R^n=\ker(A-\mathrm{Id})\oplus\operatorname{im}(A-\mathrm{Id})$, and $\ker(A-\mathrm{Id})$ is clearly $H_\pi$-invariant if $A\in\centr(H_\pi)$.
\end{remark}

The hypothesis that $\pi$ is torsion-free is essential in Theorem~\ref{thm:redholonomy}. In fact, it is easy to find crystallographic groups $\pi\subset\Iso(\R^n)$ whose holonomy $H_\pi$ acts irreducibly on $\R^n$, see Subsection~\ref{subsec:flat2orbifolds} for examples with $n=2$.

\section{Sequences of flat manifolds}\label{sec:sequences}

In this section, we analyze sequences of closed flat manifolds, proving Theorem~\ref{mainthm:limits}.

\subsection{Gromov-Hausdorff distance}
A map $f\colon X\to Y$ between metric spaces, not necessarily continuous, is called an $\varepsilon$-approximation if an $\varepsilon$-neighborhood of its image covers all of $Y$ and $|d_X(p,q)-d_Y(f(p),f(q))|\leq\varepsilon$ for all $p,q\in X$.
The \emph{Gromov-Hausdorff distance} 
between two compact metric spaces $X$ and $Y$ is the infimum of $\varepsilon>0$ such that there exist $\varepsilon$-approximations $X\to Y$ and $Y\to X$. 
This distance function between (isometry classes of) compact metric spaces and the corresponding notion of convergence were pioneered by Gromov~\cite{Gromov81}.

Gromov-Hausdorff convergence can be easily extended to \emph{pointed} complete metric spaces, by declaring that $(X_i,p_i)$ converges to $(X,p)$ if, for all $r>0$, the ball of radius $r$ in $X_i$ centered at $p_i$ Gromov-Hausdorff converges to the ball of radius $r$ in $X$ centered at $p$. Furthermore, an \emph{equivariant} extension of this notion was introduced by Fukaya~\cite{Fuka85} and achieved its final form with Fukaya and Yamaguchi~\cite{FukaYama92}.

\subsection{Converging sequences of flat manifolds}
We begin by analyzing the case of flat tori:

\begin{proposition}\label{prop:toritotori}
The Gromov-Hausdorff limit of a sequence $\{(T^n,\g_i)\}_{i\in\N}$ of flat tori with bounded diameter is a flat torus $(T^m,\h)$ of dimension $0\leq m\leq n$.
\end{proposition}

\begin{proof}
The above result follows from a general construction of limits of isometric group actions due to Fukaya and Yamaguchi~\cite[Prop.~3.6]{FukaYama92}. Namely, let $\{(X_i,p_i)\}_{i\in\N}$ be a sequence of pointed complete Riemannian manifolds (more generally, pointed locally compact length spaces) that Gromov-Hausdorff converges to a limit space $(X,p)$ and $G_i\subset\Iso(X_i)$ be closed subgroups of isometries. Then there is a closed subgroup $G\subset\Iso(X)$ such that $\{(X_i,G_i,p_i)\}_{i\in\N}$ converges in equivariant Gromov-Hausdorff sense to $(X,G,p)$. In particular, the orbit spaces $\{(X_i/G_i,[p_i])\}_{i\in\N}$ Gromov-Hausdorff converge to the orbit space $(X/G,[p])$.

Each flat torus $(T^n,\g_i)$ is isometric to $\R^n/G_i$, where $G_i\subset\Iso(\R^n)$ is a lattice, that is, a discrete subgroup consisting only of translations. Applying the aforementioned result to the constant sequence $(X_i,p_i)=(\R^n,0)$ and the lattices $G_i$, it follows that there exists a closed subgroup $G\subset\Iso(\R^n)$ such that $(T^n,\g_i)$ converge in Gromov-Hausdorff sense to $\R^n/G$. 
We claim that this limit group $G\subset\Iso(\R^n)$ is a degenerate lattice, that is, $G\cong L\times\R^{n-m}$ for some $0\leq m\leq n$, where $L\subset\Iso(\R^m)$ is a lattice, hence $\R^n/G$ is isometric to the flat torus $T^m=\R^m/L$.

To prove this claim, which concludes the proof of the Proposition, we use that \emph{metric} properties of the isometries $G_i$ are preserved in the limit $G$ by its inductive-projective construction. Translations are metrically characterized as isometries that have constant displacement, i.e., isometries that move all points in $\R^n$ by the same distance. Since all elements of $G_i$ satisfy this property, for all $i\in\N$, also all elements of $G$ satisfy it and are hence translations.
Moreover, the only \emph{closed} subgroups $G$ of $\Iso(\R^n)$ that consist of translations are degenerate lattices in subspaces of $\R^n$. Indeed, the identity connected component $G_0\cong\R^{n-m}$ is a subspace of $\R^n$ and the quotient $L=G/\R^{n-m}$ is discrete and abelian, hence a lattice in a subspace $V\subset\R^n$ with $V\cap \R^{n-m}=\{0\}$. As the limit space $\R^n/G$ is compact, since $\R^n/G_i$ have bounded diameter, we have that $V\cong\R^{m}$ is a complement of $G_0\cong\R^{n-m}$. This proves that $G\cong L\times\R^{n-m}$ is a degenerate lattice in $\R^n$.
\end{proof}

By the Bieberbach Theorems, there are only finitely many diffeomorphism types of closed flat $n$-manifolds for any given $n\in\N$. Thus, up to subsequences, we may assume that any Gromov-Hausdorff converging sequence of closed flat $n$-manifolds is of the form $\{(M,\g_i)\}_{i\in\N}$, where $\g_i$ are flat metrics on a fixed manifold~$M$.

\begin{proposition}\label{prop:limitaction}
Let $\{(M,\g_i)\}_{i\in\N}$ be a Gromov-Hausdorff sequence of closed flat $n$-manifolds that converges to a limit metric space $(X,d_X)$. Then $(X,d_X)$ is isometric to a flat orbifold $T^m/H$, where $0\leq m\leq n$ and $H\subset\O(n)$ is a finite subgroup conjugate to the holonomy group of $(M,\g_i)$ that acts isometrically on $T^m$.
\end{proposition}

\begin{proof}
Without loss of generality, we may assume that the holonomy groups of $(M,\g_i)$ are all equal to $H\subset\O(n)$, see Corollary~\ref{thm:sameholonomy}.
Let $(T^n,\widetilde \g_i)\to (M,\g_i)$ be the Riemannian coverings by flat tori whose group of deck transformations is $H$. Then,
\[\frac1{\vert H\vert}\,\operatorname{diam}(T^n,\widetilde \g_i)\le\operatorname{diam}(M,\g_i)\le\operatorname{diam}(T^n,\widetilde \g_i).\]
Therefore, $\operatorname{diam}(T^n,\widetilde\g_i)\leq 2|H|\operatorname{diam}(X,d_X)$ and, by Gromov's Compactness Theorem and Proposition~\ref{prop:toritotori}, there is a subsequence of $\{(T^n,\widetilde \g_i)\}_{i\in\N}$ that Gromov-Hausdorff converges to a flat torus $(T^m,\h)$, $0\leq m\leq n$.
Up to passing to a new subsequence, this convergence is in equivariant Gromov-Hausdorff sense with respect to the isometric $H$-actions, by Fukaya and Yamaguchi~\cite[Prop.~3.6]{FukaYama92}.
Thus, the orbit spaces $\{(T^n,\widetilde\g_i)/H\}_{i\in\N}$, which are isometric to $\{(M,\g_i)\}_{i\in\N}$, Gromov-Hausdorff converge to the orbit space $(T^m,\h)/H$.
\end{proof}

To finish the proof of Theorem~\ref{mainthm:limits}, it  only remains to prove the following:

\begin{proposition}
Any flat orbifold is the Gromov-Hausdorff limit of a sequence of closed flat manifolds.
\end{proposition}

\begin{proof}
Let $\mathcal O=\R^n/\pi_{\mathcal O}$ be a flat orbifold, with orbifold fundamental group given by the crystallographic group $\pi_{\mathcal O}\subset\Iso(\R^n)$. By Theorem~\ref{thm:AKthm} (of Auslander and Kuranishi~\cite{auslanderkuranishi57}), there exists a closed flat manifold $M=\R^m/\pi_M$ whose fundamental group $\pi_M\subset\Iso(\R^m)$ is a Bieberbach group with the same holonomy, that is, $H_{\pi_M}\cong H_{\pi_{\mathcal O}}$. Denote this finite group by $H$, and consider its (isometric) actions on the flat tori $T^n=\R^n/L_{\pi_{\mathcal O}}$ and $T^m=\R^m/L_{\pi_M}$. Clearly, the orbit spaces of these actions are $T^n/H=\mathcal O$ and $T^m/H=M$. Since the $H$-action on $T^m$ is free, so is the diagonal $H$-action on the product $T^n\times T^m$. Thus, $N=(T^n\times T^m)/H$ is a closed flat manifold. The product metrics $\g_\lambda=\g_{T^n}\oplus\lambda\,\g_{T^m}$, $\lambda>0$, are invariant under the $H$-action and hence descend to flat metrics on $N$. The closed flat manifolds $\{(N,\g_\lambda)\}_{\lambda>0}$ clearly Gromov-Hausdorff converge to $\mathcal O$ as $\lambda\searrow0$.
\end{proof}

\section{Teichm\"uller space and moduli space of flat metrics}\label{sec:algdescr}

In this section, we study the Teichm\"uller space and moduli space of flat metrics on a closed flat orbifold (or manifold), and prove Theorem~\ref{mainthm:teichdescription} in the Introduction.

\subsection{Teichm\"uller space}
It follows from the Bieberbach Theorems that any flat metrics on a given closed flat orbifold (or manifold) $\mathcal O=\R^n/\pi$ are of the form $\g_{\pi'}$ where $\pi'=(A,v)\,\pi\,(A,v)^{-1}$ for some $(A,v)\in\aff(\R^n)$. In this situation,
\begin{equation}\label{eq:conjholonomy}
H_{\pi'}=A\,H_\pi\, A^{-1}\quad\text{and}\quad L_{\pi'}=A(L_\pi).
\end{equation}

Distinguishing isometry classes of such metrics is straightforward (see also \cite{BetPic2016}):

\begin{lemma}\label{lemma:isomiff}
The metrics $\g_\pi$ and $\g_{\pi'}$ are isometric, where  $\pi'=(A,v)\,\pi\,(A,v)^{-1}$, if and only if $A=B\,C$, where $B\in\O(n)$ and $C\in\mathcal N_\pi:=\rot\big(\!\norm_{\aff(\R^n)}(\pi)\big)$.
\end{lemma}

\begin{proof}
By lifting isometries to $\R^n$, it is clear that $\g_\pi$ and $\g_{\pi'}$ are isometric if and only if there is $(B,w)\in\Iso(\R^n)$ such that $(B,w)\,\pi\,(B,w)^{-1}=(A,v)\,\pi\,(A,v)^{-1}$, that is, $(C,z):=(B,w)^{-1}(A,v)\in\norm_{\aff(\R^n)}(\pi)$. If $\g_\pi$ and $\g_{\pi'}$ are isometric, then clearly $A=B \, C$, with $B\in\O(n)$  and $C\in\mathcal N_\pi$. Conversely, if $A=B \, C$, with $B\in\O(n)$ and $C\in\mathcal N_\pi$, there is $z\in\R^n$ such that $(C,z)\in\norm_{\aff(\R^n)}(\pi)$. Set $w=v-Bz$, so that  $(A,v)=(B,w)\, (C,z)$. Clearly, $(B,w)^{-1}(A,v)\in\norm_{\aff(\R^n)}(\pi)$, so $\g_\pi$ and $\g_{\pi'}$ are isometric.
\end{proof}

While Lemma~\ref{lemma:isomiff} provides the appropriate equivalence relation to distinguish isometry classes, it remains to characterize the space where these relations take place. For a given crystallographic group $\pi\subset\Iso(\R^n)$, we let
\begin{equation}\label{eq:cpi}
\mathcal C_\pi:=\big\{A\in\GL(n):A\,H_\pi \, A^{-1}\subset\O(n)\big\}.
\end{equation}
It is easy to see that $A\,H_\pi \, A^{-1}\subset\O(n)$ is equivalent to $A^{\mathrm t}A\in\centr_{\GL(n)}(H_\pi)$, where $A^{\mathrm t}$ is the transpose of $A$. The set $\mathcal C_\pi$ is a closed cone in $\GL(n)$ that contains $\norm_{\GL(n)}(H_\pi)$. There are natural actions on $\mathcal C_\pi$ by matrix multiplication, on the left by $\O(n)$ and on the right by $\norm_{\GL(n)}(H_\pi)$. In this context, it is natural to introduce the following:

\begin{definition}\label{thm:classesorbits}
The \emph{Teichm\"uller space} $\teichflat(\mathcal O)$ of the flat orbifold $\mathcal O=\R^n/\pi$ is the orbit space $\O(n)\backslash\mathcal C_\pi$ of the left $\O(n)$-action on $\mathcal C_\pi$.
\end{definition}

The Teichm\"uller space $\teichflat(\mathcal O)$ is a real-analytic manifold diffeomorphic to $\R^d$, and is described further in Subsection~\ref{subsec:algebraicdescrpt}.
It can also be obtained as the deformation space of certain $(X,G)$-structures, see Section~\ref{sub:defXG}.

\subsection{Moduli space}
The \emph{moduli space} $\Flat(\mathcal O)$ of the flat orbifold $\mathcal O=\R^n/\pi$ is defined as the set of isometry classes of flat metrics on $\mathcal O$.
Considering the restriction to $\mathcal N_\pi$ of the right $\norm_{\GL(n)}(H_\pi)$-action on $\mathcal C_\pi$, the following identification follows directly from Lemma~\ref{lemma:isomiff}, cf.~Wolf~\cite[Thm.~1]{Wolf73}.

\begin{proposition}\label{thm:discrquotient}
$\Flat(\mathcal O)=\teichflat(\mathcal O)/\mathcal N_\pi$.
\end{proposition}

Note that the right $\mathcal N_\pi$-action on $\mathcal C_\pi$ need not be free, so $\Flat(\mathcal O)$ may have (isolated) singularities; this may also be the case if $\mathcal O$ is a smooth manifold. As indicated by Proposition~\ref{thm:discrquotient}, the group $\mathcal N_\pi$ is related to the \emph{mapping class group} in this Teichm\"uller theory (see Remark~\ref{rmk:mcg}), and moreover satisfies the following:

\begin{proposition}
$\mathcal N_\pi$ is isomorphic to a discrete subgroup of the group $\aff(\mathcal O)\cong\norm_{\aff(\R^n)}(\pi)/\centr_{\aff(\R^n)}(\pi)$ of affine diffeomorphisms of $\mathcal O=\R^n/\pi$.
\end{proposition}

\begin{proof}
It is easy to see that $\mathcal N_\pi$ is a countable (Lie) group, and thus discrete. In fact, if $\pi$ is a lattice, then $\mathcal N_\pi$ is a conjugate of $\GL(n,\Z)$ inside $\GL(n)$ and hence countable; while for a general crystallographic group $\pi$, one has $\mathcal N_\pi\subset\mathcal N_{L_\pi}$.

A diffeomorphism $\phi\colon \mathcal O\to \mathcal O$ that preserves the affine structure of the flat orbifold $\mathcal O=\R^n/\pi$ endowed with the metric $\g_\pi$ lifts to an affine diffeomorphism $\widetilde\phi\colon\R^n\to\R^n$. Conversely, an affine diffeomorphism $\widetilde\phi$ of $\R^n$ descends to an affine diffeomorphism $\phi$ of $(\mathcal O,\g_\pi)$ if and only if $\widetilde\phi$ normalizes $\pi$. We thus have a surjective homomorphism from the normalizer of $\pi$ in $\aff(\R^n)$ to the group $\aff(\mathcal O)$ of affine diffeomorphisms of $(\mathcal O,\g_\pi)$, given by $\norm_{\aff(\R^n)}(\pi)\ni\widetilde\phi\mapsto\phi\in\aff(\mathcal O)$, whose kernel is the centralizer $\centr_{\aff(\R^n)}(\pi)$. This establishes the isomorphism:
\begin{equation*}
\aff(\mathcal O)\cong\norm_{\aff(\R^n)}(\pi)/\centr_{\aff(\R^n)}(\pi).
\end{equation*}
Note that the group of affine diffeomorphisms of the closed flat orbifold $\mathcal O=\R^n/\pi$ does not depend on the flat metric, but only on the isomorphism class of the crystallographic group $\pi$, by the Bieberbach Theorems.

The group $\mathcal N_\pi=\rot\big(\!\norm_{\aff(\R^n)}(\pi)\big)$ is isomorphic to the quotient of $\norm_{\aff(\R^n)}(\pi)$ by the kernel of the projection homomorphism $\rot\colon\norm_{\aff(\R^n)}(\pi)\to\norm_{\GL(n)}(H_\pi)$, which is given by $\norm_{\aff(\R^n)}(\pi)\cap(\{\mathrm{Id}\}\times\R^n)=\big\{(\mathrm{Id},w):Aw-w\in L_\pi,\ \mbox{for all}\,A\in H_\pi\big\}$.
Clearly, the latter contains $\centr_{\aff(\R^n)}(\pi)=\big\{(\mathrm{Id},w): Aw=w,\ \mbox{for all}\,A\in H_\pi\big\}$, and hence $\mathcal N_\pi$ is isomorphic to a subgroup of $\aff(\mathcal O)$, concluding the proof.
\end{proof}

\begin{remark}
In general, $\mathcal N_\pi$ is a \emph{proper} subgroup of $\aff_{\text{flat}}(\mathcal O)$. For instance, if $\R^2/\pi$ is the Klein bottle, then $L_\pi=\Z^2$, and the orthogonal representation of $H_\pi\cong\Z_2$ on $\R^2$ is by reflection about the $x$-axis, so $\centr_{\aff(\R^2)}(\pi)=\big\{(\mathrm{Id},w): w=(w_1,0)\in \Z^2 \big\}$, while $\norm_{\aff(\R^2)}(\pi)\cap(\{\mathrm{Id}\}\times\R^2)=\big\{(\mathrm{Id},w):w=(w_1,w_2)\in \Z^2\big\}$.
\end{remark}

\subsection{\texorpdfstring{Deformations of $\textbf{(X,G)}$-structures}{Deformations of (X,G)-structures}}\label{sub:defXG}
The Teichm\"uller space $\teichflat(M)$ and the moduli space $\Flat(M)$ of a flat manifold can also be described using the language of $(X,G)$-structures~\cite{thurston}. Given a Lie group $G$ and a $G$-homogeneous space $X$, an \emph{$(X,G)$-structure} on a manifold $M$ is a (maximal) atlas of charts on $M$ with values in $X$, whose transition maps are given by restrictions of elements of $G$. Setting $X=\R^n$ and $G=\Iso(\R^n)$, an $(X,G)$-structure on a manifold $M$ is precisely a flat Riemannian metric on $M$.

A general deformation theory of $(X,G)$-structures on a given manifold $M$ is discussed in~\cite{Baues2000,Goldman88,thurston}. 
Using the action of the diffeomorphism group $\Diff(M)$ on the space of $(X,G)$-structures on $M$, one defines the corresponding \emph{deformation space} and the \emph{moduli space} respectively as
\begin{equation*}
\mathcal D(M)=\{(X,G)\text{-structures on }M\}/\Diff_0(M), \qquad \mathcal M(M)=\{(X,G)\text{-structures on }M\}/\Diff(M),
\end{equation*}
where $\Diff_0(M)\subset\Diff(M)$ is the connected component of the identity. 
The \emph{mapping class group} of $M$ is defined as $\MCG(M)=\Diff(M)/\Diff_0(M)$, so that $\mathcal M(M)=\mathcal D(M)/\MCG(M)$.

\begin{proposition}
For $(X,G)=\big(\R^n,\Iso(\R^n)\big)$, the deformation space $\mathcal D(M)$ and the moduli space $\mathcal M(M)$ can be respectively identified with the Teichm\"uller space $\teichflat(M)$ and the moduli space $\Flat(M)$.
\end{proposition}

\begin{proof}
Assume $M=\R^n/\pi$, where $\pi\subset\Iso(\R^n)$ is a Bieberbach group, and consider the set $\Inj\big(\pi,\Iso(\R^n)\big)$ of injective homomorphisms of $\pi$ into $\Iso(\mathds R^n)$. There is a left action of $\Iso(\R^n)$ on $\Inj\big(\pi,\Iso(\R^n)\big)$ by conjugation, i.e., composition with inner automorphisms. The deformation space $\mathcal D(M)$ is identified with the quotient $\Iso(\R^n)\backslash\Inj\big(\pi,\Iso(\R^n)\big)$, see \cite[Prop.~1.6]{Baues2000}. By the Bieberbach Theorems, $\Inj\big(\pi,\Iso(\R^n)\big)$ can be identified with $\mathcal C_\pi\times\R^n$. Using this identification, the action of $\Iso(\R^n)=\O(n)\ltimes\R^n$ on $\Inj\big(\pi,\Iso(\R^n)\big)$ is given by left multiplication, hence the quotient $\mathcal D(M)=\Iso(\R^n)\backslash\Inj\big(\pi,\Iso(\R^n)\big)$ is identified with the quotient $\teichflat(M)=\O(n)\backslash \mathcal C_\pi$, cf.~Definition~\ref{thm:classesorbits}.

By a generalization of the Dehn-Nielsen-Baer Theorem to flat manifolds, the action of $\Diff(M)$ on the loop space of $M$ induces an isomorphism from $\MCG(M)$ to the group $\Out(\pi)$ of outer automorphisms of $\pi$. Moreover, $\Out(\pi)$ is isomorphic to $\mathrm{Aff}(M)/\mathrm{Aff}_0(M)$, see \cite[Thm.~6.1]{charlap}. 
Using \cite[Lemma~6.1]{charlap}, 
it follows that the orbits of the action of this group on $\Iso(\R^n)\backslash\Inj\big(\pi,\Iso(\R^n)\big)$ are the same as the orbits of $\norm_{\aff(\R^n)}(\pi)$ acting by composition on the right with conjugations.
Under the identification of $\mathcal D(M)=\Iso(\R^n)\backslash\Inj\big(\pi,\Iso(\R^n)\big)$ with $\teichflat(M)=\O(n)\backslash\mathcal C_\pi$, the action of $\norm_{\aff(\R^n)}(\pi)$ coincides with the action of the group $\mathcal N_\pi$ (see Lemma~\ref{lemma:isomiff}) by right multiplication. Thus, $\mathcal M(M)$ is identified with $\Flat(M)$, cf.~Proposition~\ref{thm:discrquotient}.
\end{proof}

\begin{remark}\label{rmk:mcg}
By the results quoted above, given a flat manifold $M=\R^n/\pi$, there are isomorphisms
\begin{equation}\label{eq:mcg}
\Diff(M)/\Diff_0(M)=\MCG(M)\cong\Out(\pi)\cong \mathrm{Aff}(M)/\mathrm{Aff}_0(M).
\end{equation}
\end{remark}

\subsection{Algebraic description of Teichm\"uller space}\label{subsec:algebraicdescrpt}
Let $\mathcal O=\R^n/\pi$ be a closed flat orbifold and $\teichflat(\mathcal O)$ be its Teichm\"uller space. We now employ simple algebraic considerations to identify $\teichflat(\mathcal O)$ with a product of (noncompact) homogeneous spaces, proving Theorem~\ref{mainthm:teichdescription}.

\begin{proposition}\label{thm:charXpi}
The cone $\mathcal C_\pi$ defined in \eqref{eq:cpi} satisfies $\mathcal C_\pi=\O(n)\cdot\centr_{\GL(n)}(H_\pi)$.
\end{proposition}
 
\begin{proof}
Clearly, $\O(n)\cdot\centr_{\GL(n)}(H_\pi)\subset\mathcal C_\pi$.
Choose $A\in\mathcal C_\pi$, and write the polar decomposition $A=OP$, with $P=(A^{\mathrm t}A)^\frac12$ and $O\in\O(n)$. Since $A^{\mathrm t}A$ centralizes $H_\pi$, also $P$ centralizes $H_\pi$. Thus, $A\in\O(n)\cdot\centr_{\GL(n)}(H_\pi)$, concluding the proof.
\end{proof}

The following is an immediate consequence of Proposition~\ref{thm:charXpi}  and Definition~\ref{thm:classesorbits}.

\begin{corollary}\label{cor:teichcyclicequivrel}
The Teichm\"uller space $\teichflat(\mathcal O)$ is identified with the quotient space $\centr_{\GL(n)}(H_\pi)/\simeq$, where $A\simeq B$ if there exists $O\in\O(n)$ with $B=O\cdot A$. Thus, it can be described as the space of right cosets $(\centr_{\GL(n)}(H_\pi)\cap\O(n))\backslash\centr_{\GL(n)}(H_\pi)$.
\end{corollary}

According to \eqref{eq:conjholonomy}, different flat metrics on the same closed orbifold have conjugate holonomy groups. A more precise statement follows from Proposition~\ref{thm:charXpi}:

\begin{corollary}\label{thm:sameholonomy}
Any flat metric $\g$ on $\mathcal O=\R^n/\pi$ is isometric to a (flat) metric on $\mathcal O$ that has the same holonomy as $\g_\pi$.
\end{corollary}

\begin{proof}
There is a crystallographic group $\pi'\subset\Iso(\R^n)$ such that $\g=\g_{\pi'}$. By Proposition~\ref{thm:charXpi}, there exist $O\in\O(n)$ and 
$A\in\centr_{\GL(n)}(H_\pi)$ such that $\pi'=OA\cdot\pi\cdot A^{-1}O^\mathrm t$. Thus, $\g_{\pi'}$ is isometric to $\g_{\pi''}$, where $\pi''=A\cdot\pi\cdot A^{-1}$, and $H_{\pi''}=H_\pi$ since $A$ normalizes $H_\pi$.
\end{proof}

In order to achieve a more precise description of $\centr_{\GL(n)}({H_\pi})$, leading to the proof of Theorem~\ref{mainthm:teichdescription} via Corollary~\ref{cor:teichcyclicequivrel}, we use the decomposition of $\R^n$ into $H_\pi$-isotypic components. By a result of Hiss and Szczepa\'nski~\cite{hiss-szczepa}, see Theorem~\ref{thm:redholonomy}, there are always nontrivial invariant subspaces of the orthogonal $H_\pi$-rep\-resentation on $\R^n$. Recall that a nonzero invariant subspace $V$ is \emph{irreducible} if it contains no proper invariant subspaces; or, equivalently, if every nonzero element of the vector space $\mathrm{End}_{H_\pi}(V)$ of linear equivariant endomorphisms of $V$ is an isomorphism. In this situation, $\mathrm{End}_{H_\pi}(V)$ is an associative real division algebra, hence isomorphic to one of $\R$, $\C$, or $\H$. The irreducible $V$ is called of \emph{real}, \emph{complex}, or \emph{quaternionic} type, according to the isomorphism type of $\mathrm{End}_{H_\pi}(V)$.

Consider the decomposition of the orthogonal $H_\pi$-representation into irreducibles,
\begin{equation}\label{eq:Phiisotypicdec}
\R^n=\underbrace{V_{1,1}\oplus\ldots\oplus V_{1,m_1}}_{W_1}\oplus\ldots\oplus \underbrace{V_{l,1}\oplus\ldots\oplus V_{l,m_l}}_{W_l},
\end{equation}
where each $V_{i,j}$ is irreducible and $V_{i,j}$ is isomorphic to $V_{i',j'}$ if and only if $i=i'$, so that $W_i=\bigoplus_{j=1}^{m_i}V_{i,j}$, $i=1,\ldots,l$ are the so-called \emph{isotypic components}.
Let $\mathds K_i$ be the real division algebra $\R$, $\C$, or $\H$, according to $W_i$ consisting of irreducibles $V_{i,j}$ of real, complex, or quaternionic type.

As $\centr_{\GL(n)}(H_\pi)=\mathrm{End}_{H_\pi}(\R^n)\cap\GL(n)$, and $\mathrm{End}_{H_\pi}(\R^n)\cong\prod_{i=1}^l\mathrm{End}_{H_\pi}(W_i)$ by Schur's Lemma, it follows that there is an isomorphism:
\begin{equation}\label{eq:centrPhi}
\centr_{\GL(n)}({H_\pi})\cong\prod_{i=1}^l\GL(m_i,\mathds K_i).
\end{equation}
For each isotypic component $W_i$, it is clear that $\mathrm{End}_{H_\pi}(W_i)\cap \O(W_i)\cong \O(m_i,\mathds K_i)$,
where
\begin{equation*}
\O(m,\mathds K) :=\begin{cases}\O(m),&\text{if $\mathds K=\R$;}\\
\U(m),&\text{if $\mathds K=\C$;}\\ \Sp(m),&\text{if $\mathds K=\H$.}\end{cases}
\end{equation*}
Thus, from Corollary~\ref{cor:teichcyclicequivrel} and these isomorphisms,
\begin{equation*}
\teichflat(\mathcal O)\cong\frac{\centr_{\GL(n)}({H_\pi})}{\centr_{\GL(n)}({H_\pi})\cap\O(n)}\cong\frac{\prod_{i=1}^l \mathrm{End}_{H_\pi}(W_i)\cap \GL(W_i)}{\prod_{i=1}^l\mathrm{End}_{H_\pi}(W_i)\cap \O(W_i)}\cong\prod_{i=1}^l\frac{\GL(m_i,\mathds K_i)}{\O(m_i,\mathds K_i)},
\end{equation*}
which concludes the proof of Theorem~\ref{mainthm:teichdescription} in the Introduction.

\section{Examples of Teichm\"uller spaces}\label{sec:exa}

In this section, we apply Theorem~\ref{mainthm:teichdescription} to compute the Teichm\"uller space of some flat manifolds and orbifolds.

\subsection{Flat tori}\label{subsec:tori}
If $\pi=L_\pi$ is a lattice, then $H_\pi$ is trivial and $M=\R^n/\pi$ is a flat torus $T^n$. In this case, Theorem~\ref{mainthm:teichdescription} gives:
\begin{equation}\label{eq:teichflattorus}
\teichflat(T^n)\cong\frac{\GL(n)}{\O(n)}\cong\R^{n(n+1)/2}.
\end{equation}
Furthermore, it is easy to see that $\norm_{\aff(\R^n)}(\pi)=\R^n\ltimes\GL(n,\Z)$, so $\mathcal N_\pi\cong\GL(n,\Z)$ is the discrete subgroup of integer matrices with determinant $\pm1$. Thus, $\Flat(T^n)=\O(n)\backslash \GL(n)/\GL(n,\Z)$, cf.~Wolf~\cite[Cor]{Wolf73}.

\begin{remark}
Flat metrics on the $2$-torus $T^2$ are often parametrized using the upper half plane, identifying each flat metric $\g_\pi$ with $w/z\in\C$, where $z,w\in\C$ are chosen so that $\pi=\operatorname{span}_\Z\{z,w\}$, $z\in\R$, $z>0$, and $\operatorname{Im}w>0$. This parametrization clearly identifies homothetic metrics, while they are distinct in $\teichflat(T^2)$ and $\Flat(T^2)$.

The moduli space $\Flat(T^2)=\teichflat(T^2)/\GL(2,\Z)$ has two singular strata of dimension $1$, corresponding to hexagonal and square lattices. Furthermore, it has one end whose boundary at infinity is a ray $[0,+\infty)$, corresponding to the lengths of circles to which $T^2$ can collapse. For details, see \cite[\S 12.1]{fm}.
\end{remark}

\subsection{Cyclic holonomy}
Assume that $H_\pi\subset\O(n)$ is a cyclic group, and choose a generator $A\in H_\pi$.
Up to rewriting $A$ in its real canonical form, we may assume that it is block diagonal, with $1\times 1$ and $2\times 2$ blocks, where $m_1$ diagonal entries are $1$, $m_2$ diagonal entries are $-1$, and each of the $2\times 2$ blocks is a rotation matrix
\begin{equation}\label{eq:rotmatrix}
\phantom{,\qquad \theta_i\in (0,\pi), \; i=3, \dots, l,}
\begin{pmatrix*}[r]
\cos \theta_i&-\sin\theta_i\\
\sin\theta_i&\cos\theta_i
\end{pmatrix*},\qquad \theta_i\in (0,\pi), \; i=3, \dots, l,
\end{equation}
appearing $m_i$ times.
Therefore, the decomposition \eqref{eq:Phiisotypicdec} has $2$ isotypic components of real type with dimensions $m_1$ and $m_2$, namely, $W_1=\ker(A-\mathrm{Id})\neq\{0\}$, and $W_2=\ker(A+\mathrm{Id})$, and all other isotypic components $W_i$, $i=3, \dots, l$, are of complex type and consist of $m_i$ copies of $\R^2$.

The centralizer $\centr_{\GL(n)}(H_\pi)$ consists of block diagonal matrices whose first two blocks (corresponding to $W_1$ and $W_2$) are any invertible $m_1\times m_1$ and $m_2\times m_2$ matrices, while the remaining blocks are $m_i\times m_i$ invertible matrices which are also complex linear.
In other words, there is an isomorphism
\[\centr_{\GL(n)}(H_\pi)\cong\GL(m_1,\R)\times\GL(m_2,\R)\times\prod_{i=3}^{l}\GL(m_i,\C).\]
Therefore, according to Theorem~\ref{mainthm:teichdescription}, the Teichm\"uller space of $M=\R^n/\pi$ is:
\begin{equation*}
\teichflat(M)\cong\frac{\GL(m_1,\R)}{\O(m_1)}\times \frac{\GL(m_2,\R)}{\O(m_2)}\times \prod_{i=3}^l \frac{\GL(m_i,\C)}{\U(m_i)}\cong\R^d,
\end{equation*}
where $d=\tfrac12m_1(m_1+1)+\tfrac12m_2(m_2+1)+\sum_{i=3}^{l}m_{i}^2$.

The case in which $H_\pi\cong\Z_p$, with $p$ prime, is particularly interesting~\cite[\S~IV.7]{charlap}. If $p=2$, then all $H_\pi$-isotypic components are of real type and hence $\teichflat(M)\cong\R^d$, with $d=\frac12m_1(m_1+1)+\tfrac12 m_2(m_2+1)$. For instance, if $M=\R^2/\pi$ is the Klein bottle, then $m_1=m_2=1$ and $\teichflat(M)\cong\R^2$.
If $p>2$, then the angles in \eqref{eq:rotmatrix} are $\theta_i=\frac{2\pi q_i}{p}$, for some $q_i\in\Z$, $i=1,\dots,p-1$, with $q_i\not\equiv q_j\mod p$. A special case is that of so-called \emph{generalized Klein bottles}~\cite[\S~IV.9]{charlap}, which have dimension $p$, and holonomy $H_\pi\cong\Z_p\subset\O(p)$ generated by a matrix with characteristic polynomial $\lambda^p-1$. Thus, in this case $\theta_i=\frac{2\pi i}{p}$ for each $i=1,\dots,p-1$ and all $m_i=1$ except for $m_2=0$, so $l=\frac{1}{2}(p-1)$ and $\teichflat(M)\cong\R^{\frac{p+1}{2}}$, cf.~\cite[Thm~9.1, p.~165]{charlap}.

\subsection{\texorpdfstring{Flat $2$-orbifolds}{Flat two-orbifolds}}\label{subsec:flat2orbifolds}
As listed in Table~\ref{table:flat2orbifolds}, there are 17 affine equivalence classes of flat $2$-orbifolds. Their Teichm\"uller spaces can be computed using Theorem~\ref{mainthm:teichdescription} as follows:
\begin{enumerate}
\item The Klein bottle, M\"obius band, and cylinder have holonomy group $\Z_2$ generated by reflection about an axis, and hence have $2$ inequivalent $1$-dimensional irreducibles. Thus, in these cases, $\teichflat(\mathcal O)=\GL(1,\R)/\O(1)\times\GL(1,\R)/\O(1)\cong\R^2$. The case of the $2$-torus is discussed in \eqref{eq:teichflattorus};
\item The flat $2$-orbifolds $D^2(;3,3,3)$, $D^2(;2,3,6)$, $D^2(;2,4,4)$, $D^2(4;2)$, $D^2(3;3)$, $S^2(3,3,3;)$, $S^2(2,3,6;)$, and $S^2(2,4,4;)$ have irreducible holonomy representation, and hence $\teichflat(\mathcal O)=\GL(1,\R)/\O(1)\cong\R$;
\item The flat $2$-orbifolds $D^2(;2,2,2,2)$, $D^2(2;2,2)$, $D^2(2,2;)$, and $\R P^2(2,2;)$ have holonomy group $D_2\cong\Z_2\oplus\Z_2$, generated by reflections about the coordinate axes, and hence $\teichflat(\mathcal O)=\GL(1,\R)/\O(1)\times\GL(1,\R)/\O(1)\cong\R^2$;
\item The flat $2$-orbifold $S^2(2,2,2,2;)$ has holonomy group $\Z_2$, generated by the rotation of $\pi$, and hence $1$ isotypic component consisting of $2$ copies of the nontrivial $1$-dimensional representation. Thus, $\teichflat(S^2(2,2,2,2;))=\GL(2,\R)/\O(2)\cong\R^3$.
\end{enumerate}

\subsection{\texorpdfstring{Flat $3$-manifolds}{Flat three-manifolds}}\label{subsec:flat3manifolds}
As mentioned in Section~\ref{sec:prelims}, there are 10 affine equivalence classes of closed flat $3$-manifolds, described in Wolf~\cite[Thm.~3.5.5, 3.5.9]{Wolfbook}. These manifolds are labeled by the corresponding Bieberbach groups, denoted $\mathcal G_i$, $i=1,\dots,6$, in the orientable case, and $\mathcal B_i$, $i=1,\dots,4$, in the non-orientable case. Their Teichm\"uller space can be computed using Theorem~\ref{mainthm:teichdescription}, reobtaining the results of \cite{Kang06,KangKim03}:
\begin{enumerate}
\item The flat $3$-manifold corresponding to $\mathcal G_1$ is the $3$-torus, see \eqref{eq:teichflattorus};
\item The flat $3$-manifolds corresponding to $\mathcal G_3$, $\mathcal G_4$, and $\mathcal G_5$ have holonomy isomorphic to $\Z_k$, with $k=3,4,6$, respectively, generated by a block diagonal matrix $A\in\O(3)$ with one eigenvalue $1$ and a $2\times 2$ block \eqref{eq:rotmatrix} with $\theta=\frac{2\pi}{k}$. Thus, for these manifolds, $\teichflat(M)\cong\GL(1,\R)/\O(1)\times\GL(1,\C)/\U(1)\cong\R^2$;
\item The flat $3$-manifolds corresponding to $\mathcal G_6$, $\mathcal B_3$, and $\mathcal B_4$ have holonomy isomorphic to $\Z_2\oplus\Z_2$, generated by $A_1=\operatorname{diag}(1,-1,1)$ and $A_2=\operatorname{diag}(\pm1,1,-1)$. In all cases, there are $3$ inequivalent $1$-dimensional irreducibles. Thus, for such $M$, $\teichflat(M)\cong\GL(1,\R)/\O(1)\times\GL(1,\R)/\O(1)\times\GL(1,\R)/\O(1)\cong\R^3$;
\item The $3$-manifolds corresponding to $\mathcal G_2$, $\mathcal B_1$, and $\mathcal B_2$ have holonomy isomorphic to $\Z_2$, generated by a diagonal matrix $A\in\O(3)$ with eigenvalues $\pm1$, one with multiplicity $1$ and another with multiplicity $2$. Thus, the Teichm\"uller space of these manifolds is $\teichflat(M)\cong\GL(1,\R)/\O(1)\times\GL(2,\R)/\O(2)\cong\R^4$.
\end{enumerate}

Comparing the above computations, it follows that closed flat $3$-manifolds that have isomorphic holonomy groups also have diffeomorphic Teichm\"uller spaces. This coincidence, however, is of course not expected to hold in higher dimensions. 

The computation of the groups $\mathcal N_\pi$ for the above manifolds, and hence of the moduli space of flat metrics $\Flat(M)$, can be found in Kang~\cite[Thm.~4.5]{Kang06}.

\subsection{Kummer surface} 
The \emph{Kummer surface} is given by $\mathcal O=T^4/\Z_2$, where $\Z_2$ acts via the antipodal map on each coordinate of $T^4$. This is a $4$-dimensional flat orbifold with 16 conical singularities, whose desingularization is a Calabi-Yau \emph{$K3$ surface} (which is not flat, but admits Ricci flat metrics). 
The holonomy representation on $\R^4$ has one isotypic component consisting of 4 copies of the nontrivial $\Z_2$-representation. Thus, its Teichm\"uller space is $\teichflat(\mathcal O)\cong\GL(4,\R)/\O(4)\cong\R^{10}$.

\subsection{Joyce orbifolds}
There are two interesting examples of $6$-dimensional flat orbifolds that, similarly to the Kummer surface above, can be desingularized to Calabi-Yau manifolds as shown by Joyce~\cite{joyce-cy}.
The first, $\mathcal O_1=T^6/\Z_4$, also appears in the work of Vafa and Witten, and has holonomy generated by the transformation $\operatorname{diag}(-1,i,i)$ of $\C^3\cong\R^6$. Thus, its Teichm\"uller space is $\teichflat(\mathcal O_1)\cong\GL(2,\R)/\O(2)\times \GL(2,\C)/\U(2)\cong\R^7$. The second, $\mathcal O_2=T^6/\Z_2\oplus\Z_2$, has holonomy generated by the transformations $\operatorname{diag}(1,-1,-1)$ and $\operatorname{diag}(-1,1,-1)$ of $\C^3\cong\R^6$, and thus $\teichflat(\mathcal O_2)\cong\GL(2,\R)/\O(2)\times \GL(2,\R)/\O(2)\times \GL(2,\R)/\O(2)\cong\R^9$.

\section{\texorpdfstring{Classification of collapsed limits of flat $3$-manifolds}{Classification of collapsed limits of flat three-manifolds}}\label{sec:examples}

In this section, we analyze the collapsed limits of closed flat $3$-manifolds to prove Theorem~\ref{mainthm:2orbifolds}.

For completeness, let us briefly discuss the trivial situation of collapse of flat manifolds in dimensions $<3$. In dimension $1$, the only closed (flat) manifold is $S^1$, and its Teichm\"uller space is clearly $1$-dimensional. In this case, the only possible collapse is to a point.

The $2$-dimensional closed flat manifolds are the $2$-torus $T^2$ and the Klein bottle $K^2$. From Proposition~\ref{prop:toritotori}, the only possible collapsed limits of $T^2$ are a point or a circle. The Bieberbach group $\pi$ corresponding to the Klein bottle $K^2=\R^2/\pi$ is generated by a lattice $L_\pi$, whose basis $\{v_1,v_2\}$ consists of orthogonal vectors, and $\left(A,\frac12v_1\right)$, where $A$ is the reflection about the line spanned by $v_1$. Thus, $H_\pi\cong\Z_2$ has isotypic components $W_1=\operatorname{span}\{v_1\}$ and $W_2=\operatorname{span}\{v_2\}$, which are the only two possible directions along which flat metrics on $K^2$ can collapse. The limit obtained by collapsing $W_2$ is clearly $S^1$. By Proposition~\ref{prop:limitaction}, the limit obtained by collapsing $W_1$ is the orbit space of the reflection $\Z_2$-action on the circle $S^1$, which is a closed interval. Therefore, $K^2$ can collapse to a point, to a circle, or to a closed interval. Geometrically, these can be seen as shrinking the lengths of either pair of opposite sides of the rectangle with boundary identifications that customarily represents $K^2$.

To analyze the possible collapses of closed flat $3$-manifolds we proceed case by case, in the same order as in Subsection~\ref{subsec:flat3manifolds}, following the notation of Wolf~\cite[Thm.~3.5.5]{Wolfbook}. The first (trivial) case $\pi=\mathcal G_1$ is that of the $3$-torus $T^3$, which by Proposition~\ref{prop:toritotori} can only collapse to a point, to $S^1$, or to $T^2$. In the remaining cases $\R^3/\pi$, we denote by $\{v_1,v_2,v_3\}$ a basis of the lattice $L_\pi$ and identify each vector $v\in L_\pi$ with $(\mathrm{Id},v)\in\pi$. Henceforth, we ignore the case of iterated collapses, that is, if a flat manifold $M$ collapses to a flat orbifold $\mathcal O_1$ and this orbifold collapses to another flat orbifold $\mathcal O_2$, then clearly $M$ can be collapsed (directly) to $\mathcal O_2$.

\subsection{\texorpdfstring{Cases with $2$-dimensional Teichm\"uller space}{Cases with 2-dimensional Teichm\"uller space}}
There are three flat $3$-manifolds with $2$-dimensional Teichm\"uller space; namely those with Bieberbach groups $\mathcal G_3$, $\mathcal G_4$, and $\mathcal G_5$, see Subsection~\ref{subsec:flat3manifolds} (2).

\begin{example}[Case $\mathcal G_3$]\label{exa:Z3}
The basis of the lattice $L_\pi$ of the Bieberbach group $\pi=\mathcal G_3$ is such that $v_1$ is orthogonal to both $v_2$ and $v_3$, $\Vert v_2\Vert=\Vert v_3\Vert$, and $v_2$ and $v_3$ span a $2$-dimensional hexagonal lattice. This group $\pi$ is generated by $L_\pi$ and $\big(A,\frac13v_1\big)$, where $A$ fixes $v_1$ and rotates its orthogonal complement by $\frac{2\pi}{3}$. Thus, $H_\pi\cong\Z_3$ and its two isotypic components are $W_1=\operatorname{span}\{v_1\}$ and $W_2=\operatorname{span}\{v_2,v_3\}$. Collapsing $W_2$, the limit is clearly $S^1$. From Proposition~\ref{prop:limitaction}, the limit obtained by collapsing $W_1$ is the orbit space of a $\Z_3$-action on the flat $2$-torus given by the quotient of $W_2\cong\R^2$ by the hexagonal lattice. A generator of $\Z_3$ acts on $W_2$ by (clockwise) rotation of angle $\frac{2\pi}{3}$, which leaves invariant the hexagonal lattice, and hence descends to the relevant $\Z_3$-action on $T^2$. The parallelogram in $W_2$ with vertices $0$, $v_2$, $v_3$, and $v_2+v_3$ is a fundamental domain for this action. The orbit space $T^2/\mathds Z_3$ is easily identified as the $2$-orbifold $S^2(3,3,3;)$ by analyzing the boundary identifications induced on this fundamental domain.
\end{example}

\begin{example}[Case $\mathcal G_4$]\label{exa:Z4}
The basis of the lattice $L_\pi$ of the Bieberbach group $\pi=\mathcal G_4$ consists of pairwise orthogonal vectors $v_1$, $v_2$, and $v_3$, with $\Vert v_2\Vert=\Vert v_3\Vert$. This group $\pi$ is generated by $L_\pi$ and $\left(A,\frac14v_1\right)$, where $A$ fixes $v_1$ and rotates its orthogonal complement by $\frac\pi2$. Thus, $H_\pi\cong\Z_4$ and the two isotypic components are $W_1=\operatorname{span}\{v_1\}$ and $W_2=\operatorname{span}\{v_2,v_3\}$. The limit obtained by collapsing $W_2$ is clearly $S^1$. Collapsing $W_1$, the limit is the orbit space of a $\Z_4$-action on the flat $2$-torus given by the quotient of $W_2\cong\R^2$ by the square lattice generated by $v_2$ and $v_3$. A generator of $\Z_4$ acts on $W_2$ by (clockwise) rotation of angle $\frac\pi2$, which leaves invariant the square lattice, and hence descends to the relevant $\Z_4$-action on $T^2$. The square with vertices $0$, $v_2$, $v_3$, and $v_2+v_3$ is a fundamental domain for this action, which is generated by the rotation of angle $\frac{\pi}{4}$ around its center. The orbit space $T^2/\Z_4$ is identified with $D^2(4;2)$, which has a singular point on the boundary corresponding to the orbit of the vertices of the square, and an interior singular point corresponding to the center of the square which is fixed by $\Z_4$.
\end{example}

\begin{example}[Case $\mathcal G_5$]\label{exa:Z6}
The Bieberbach group $\pi=\mathcal G_5$ is generated by the same lattice $L_\pi$ as in Example~\ref{exa:Z3} and $\left(A,\frac16v_3\right)$, where $A$ fixes $v_1$ and rotates its orthogonal complement by $\frac{\pi}{3}$. Thus, $H_\pi\cong\Z_6$ and the two isotypic components are again $W_1=\operatorname{span}\{v_1\}$ and $W_2=\operatorname{span}\{v_2,v_3\}$. The limit obtained collapsing $W_2$ is clearly~$S^1$. Collapsing $W_1$, the limit is the orbit space of a $\Z_6$-action on the flat $2$-torus given by the quotient of $W_2\cong\R^2$ by the hexagonal lattice.
A generator of $\Z_6$ acts on $W_2$ by (clockwise) rotation of angle $\frac\pi3$, which leaves invariant the hexagonal lattice, and hence descends to the relevant $\Z_6$-action on $T^2$. The equilateral triangle in $W_2$ with vertices $0$, $v_2$, and $v_3$ is a fundamental domain for this action, which is generated by the rotation of angle $\frac\pi3$ around its center. Therefore, the orbit space $T^2/\Z_6$ is identified with $D^2(3;3)$, which has a singular point on the boundary corresponding to the orbit of the vertices of the triangle, and an interior singular point corresponding to the center of the triangle which is fixed by $\Z_6$. Alternatively, this orbifold can be seen as the quotient of $S^2(3,3,3;)$ by an involution given by reflection about the \emph{equator} through one of the singular points. In fact, the subaction of $\Z_3\triangleleft \Z_6$ on $T^2$ has orbit space $T^2/\Z_3=S^2(3,3,3;)$ as described in Example~\ref{exa:Z3}, and hence the orbit space $T^2/\Z_6$ is given by $S^2(3,3,3;)/\Z_2=D^2(3;3)$.
\end{example}

Note that the flat $2$-orbifolds $S^2(3,3,3;)$, $D^2(4;2)$, and $D^2(3;3)$ that arise as collapsed limits of the above flat $3$-manifolds $M$ with $2$-dimensional Teichm\"uller space have irreducible holonomy (see Subsection~\ref{subsec:flat2orbifolds}). Thus, the Teichm\"uller space of such orbifolds is $1$-dimensional. This is in accordance with the stratification of the ideal boundary of $\teichflat(M)\cong\R^2$ by the Teichm\"uller spaces of the collapsed limits of $M$, which also include a point and $S^1$, that have $0$- and $1$-dimensional Teichm\"uller spaces respectively.

\begin{remark}
Similar conclusions regarding the ideal boundary of Teichm\"uller spaces hold in general, for instance in the higher dimensional cases discussed below, provided that two subtleties are taken into account. First, if a flat orbifold $\mathcal O$ is a collapsed limit of a flat manifold $M$, it is not necessarily true that \emph{all} flat metrics on $\mathcal O$ arise as collapsed limits of flat metrics on $M$. Thus, parts of the Teichm\"uller space of $\mathcal O$ might be absent from the stratum of the boundary of $\teichflat(M)$ that corresponds to collapse to $\mathcal O$. Second, there may be a continuum of inequivalent ways to collapse $M$ to $\mathcal O$, so that the stratum of the boundary of $\teichflat(M)$ that corresponds to collapse to $\mathcal O$ may contain a continuum of copies of $\teichflat(O)$.

Recall that the Gromov-Hausdorff distance between the metric spaces given by $M$ equipped with different flat metrics \emph{does not extend continuously} to the ideal boundary of $\teichflat(M)$, as discussed in the Introduction.
\end{remark}

\subsection{\texorpdfstring{Cases with $3$-dimensional Teichm\"uller space}{Cases with 3-dimensional Teichm\"uller space}}
There are three flat $3$-manifolds with $3$-dimensional Teichm\"uller space; namely those with Bieberbach groups $\mathcal G_6$, $\mathcal B_3$, and $\mathcal B_4$, see Subsection~\ref{subsec:flat3manifolds} (3).

\begin{example}[Case $\mathcal G_6$]\label{exa:Z2Z2}
The basis of the lattice $L_\pi$ of the Bieberbach group $\pi=\mathcal G_6$ consists of pairwise orthogonal vectors $v_1$, $v_2$ and $v_3$.
This group $\pi$ is generated by $L_\pi$ together with
\[(A,\tfrac12v_1),\quad \big(B,\tfrac12(v_1+v_2)\big),\quad \big(C,\tfrac12(v_1+v_2+v_3)\big),\]
where $A=\operatorname{diag}(1,-1,-1)$, $B=\operatorname{diag}(-1,1,-1)$, and $C=\operatorname{diag}(-1,-1,1)$ in the basis $\{v_1,v_2,v_3\}$. Thus, $H_\pi\cong\Z_2\oplus\Z_2$, and the three isotypic components are $W_i=\operatorname{span}\{v_i\}$, $i=1,2,3$.
\begin{enumerate}[{Case} (a)]
\item Collapsing $W_1$ yields a flat orbifold obtained as the quotient of $\R^2$ by the action of the group of isometries generated by $T_1(x,y)=(x+a,y)$, $T_2(x,y)=(x,y+b)$, $T_3(x,y)=(-x,-y)$, $T_4(x,y)=(x+\frac12a,-y+\frac12b)$ and $T_5(x,y)=(-x+\frac12a,y+\frac12b)$.
Here, $a=\vert v_2\vert>0$ and $b=\vert v_3\vert>0$. This orbifold can be identified with $\mathds RP^2(2,2;)$, for details see Appendix~\ref{sec:appflat2orbs} case $G=G_1$.
\item Collapsing $W_2$, we obtain the flat orbifold given by the quotient of $\R^2$ by the group of isometries generated by $T_1(x,y)=(x+a,y)$, $T_2(x,y)=(x,y+b)$, $T_3(x,y)=(x+\frac12a,-y)$, $T_4(x,y)=(-x,-y+\frac12b)$ and $T_5(x,y)=(-x+\frac12a,y+\frac12b)$. Here, $a=\vert v_1\vert>0$ and $b=\vert v_3\vert>0$. Clearly, $T_3^2=T_1$. As above, this can be identified with $\R P^2(2,2;)$.
\item Collapsing $W_3$, we obtain the flat orbifold given by the quotient of $\R^2$ by the group of isometries generated by $T_1(x,y)=(x+a,y)$, $T_2(x,y)=(x,y+b)$, $T_3(x,y)=(x+\frac12a,-y)$, $T_4(x,y)=(-x, y+\frac12b)$ and $T_5=(-x+\frac12a,-y+\frac12b)$. 
Here, $a=\vert v_1\vert>0$ and $b=\vert v_2\vert>0$. Clearly, $T_1=T_3^2$, $T_2=T_4^2$ and $T_4=T_5\circ T_3$.  As above, this can be identified with  $\R P^2(2,2;)$.
\end{enumerate}
\end{example}

\begin{example}[Case $\mathcal B_3$]\label{exa:Z2Z2no1}
The basis of the lattice $L_\pi$ of the Bieberbach group $\pi=\mathcal B_3$ consists of pairwise orthogonal vectors. The group $\pi$ is generated by $L_\pi$ together with $(A,\frac12v_1)$ and $(E,\frac12v_2)$, where $A=\operatorname{diag}(1,-1,-1)$ and $E=\operatorname{diag}(1,1,-1)$ in the basis $\{v_1,v_2,v_3\}$.
Thus, $H_\pi\cong\Z_2\times\Z_2$, and the three isotypic components are $W_i=\operatorname{span}\{v_i\}$, $i=1,2,3$.
\begin{enumerate}[{Case} (a)] 
\item Collapsing $W_1$ yields a flat orbifold obtained as the quotient of $\R^2$ by the action of the group of isometries generated by $T_1(x,y)=(x+a,y)$, $T_2(x,y)=(x,y+b)$, $T_3(x,y)=(-x,-y)$ and $T_4(x,y)=(x+\frac12a,-y)$. Here, $a=\vert v_2\vert>0$ and $b=\vert v_3\vert>0$. This flat orbifold is identified with $D^2(2,2;)$, for details see Appendix~\ref{sec:appflat2orbs} case $G=G_2$.
\item Collapsing $W_2$ yields a flat orbifold obtained as the quotient of $\R^2$ by the action of the group of isometries generated by $T_1(x,y)=(x+a,y)$, $T_2(x,y)=(x,y+b)$, $T_3(x,y)=(x,-y)$ and $T_4(x,y)=(x+\frac12a,-y)$. Here, $a=\vert v_1\vert>0$ and $b=\vert v_3\vert>0$. This quotient is identified with $S^1\times I$, for details see Appendix~\ref{sec:appflat2orbs} case $G=G_3$.
\item Collapsing $W_3$ yields a flat orbifold obtained as the quotient of $\R^2$ by the action of the group of isometries generated by $T_1(x,y)=(x+a,y)$, $T_2(x,y)=(x,y+b)$, $T_3(x,y)=(x+\frac12a,-y)$ and $T_4(x,y)=(x,y+\frac12b)$. Here, $a=\vert v_1\vert>0$ and $b=\vert v_2\vert>0$. Clearly, $T_1=T_3^2$ and $T_2=T_4^2$. This quotient is easily identified with a flat Klein bottle.
\end{enumerate}
\end{example}

\begin{example}[Case $\mathcal B_4$]\label{exa:Z2Z2no2}
The basis of the lattice $L_\pi$ of the Bieberbach group $\pi=\mathcal B_4$ also consists of pairwise orthogonal vectors. The group $\pi$ is generated by $L_\pi$ together with $(A,\frac12v_1)$ and $(E,\frac12(v_2+v_3))$, where
$A=\operatorname{diag}(1,-1,-1)$ and $E=\operatorname{diag}(1,1,-1)$ in the basis $\{v_1,v_2,v_3\}$. Thus, $H_\pi\cong\Z_2\times\Z_2$, and the three isotypic components are $W_i=\operatorname{span}\{v_i\}$, $i=1,2,3$.
\begin{enumerate}[{Case} (a)]
\item Collapsing $W_1$ yields a flat orbifold obtained as the quotient of $\R^2$ by the action of the group of isometries generated by $T_1(x,y)=(x+a,y)$, $T_2(x,y)=(x,y+b)$, $T_3(x,y)=(-x,-y)$ and $T_4(x,y)=(x+\frac12a,y+\frac12b)$. 
Here, $a=\vert v_2\vert>0$ and $b=\vert v_3\vert>0$.
 This flat orbifold is identified with $S^2(2,2,2,2;)$, for details see Appendix~\ref{sec:appflat2orbs} case $G=G_4$.
\item Collapsing $W_2$ yields a flat orbifold obtained as the quotient of $\R^2$ by the action of the group of isometries generated by $T_1(x,y)=(x+a,y)$, $T_2(x,y)=(x,y+b)$, $T_3(x,y)=(x,-y+\tfrac12b)$ and $T_4(x,y)=(x+\frac12a,-y)$. 
Here, $a=\vert v_1\vert>0$ and $b=\vert v_3\vert>0$. This flat orbifold is identified with a M\"obius band, see Appendix~\ref{sec:appflat2orbs} case $G=G_5$.
\item Collapsing $W_3$ yields a flat orbifold obtained as the quotient of $\R^2$ by the action of the group of isometries generated by
$T_1(x,y)=(x+a,y)$, $T_2(x,y)=(x,y+b)$, $T_3(x,y)=(x,y+\frac12b)$ and $T_4(x,y)=(x+\frac12a,-y)$. 
Here, $a=\vert v_1\vert>0$ and $b=\vert v_2\vert>0$. Clearly, $T_4^2=T_1$ and $T_3^2=T_2$. This flat orbifold is easily identified with a Klein bottle.
\end{enumerate}
\end{example}

\subsection{\texorpdfstring{Cases with $4$-dimensional Teichm\"uller space}{Cases with 4-dimensional Teichm\"uller space}}
There are three flat $3$-manifolds with $4$-dimensional Teichm\"uller space; namely those with Bieberbach groups $\mathcal B_1$, $\mathcal B_2$, and $\mathcal G_2$, see Subsection~\ref{subsec:flat3manifolds} (4).

\begin{example}[Case $\mathcal B_1$]\label{exa:k2s1}
The manifold in question is the product $K^2\times S^1$, which by the discussion in the beginning of this section can collapse to $K^2$, $S^1\times I$, $I$, $S^1$, and a point.
\end{example}

\begin{example}[Case $\mathcal B_2$]\label{exa:Z2no}
The basis of the lattice $L_\pi$ of the Bieberbach group $\pi=\mathcal B_2$ is such that $v_1$ and $v_2$ generate any planar lattice, while $v_3$ is a vector whose orthogonal projection on the plane spanned by $v_1$ and $v_2$ is $\frac12(v_1+v_2)$.
This group $\pi$ is generated by $L_\pi$ and $(E,\frac12v_1)$, where $E$ is the identity on the plane spanned by $v_1$ and $v_2$, and $E(v_3)=v_1+v_2-v_3$, i.e., $E=\operatorname{diag}(1,1,-1)$ in the basis $\{v_1,v_2,w\}$, where $w$ is orthogonal to $v_1$ and $v_2$.
Thus, $H_\pi\cong\Z_2$ and it has one trivial isotypic component $W_1=\operatorname{span}\{v_1,v_2\}$ and one nontrivial isotypic component $W_2=\operatorname{span}\{w\}$ isomorphic to the sign representation of $\Z_2$.

Collapsing $W_2$, the limit is clearly $T^2$. Collapsing $W_1$, the limit is an interval $I$. Collapsing a one-dimensional subspace of $W_1$ produces as limit either a M\"obius band, a Klein bottle, or an interval, depending on the slope of the subspace. More precisely, collapsing the direction of $v_1$, the limiting orbifold is the quotient of $\R^2$ by the group of isometries generated by $T_1(x,y)=(x+a,y)$, $T_2(x,y)=\big(x+\frac12a,y+b\big)$, and $T_3(x,y)=(x,-y)$. Here, $a=\vert v_2\vert>0$ and $b=\vert v_3\vert>0$. This orbifold can be identified with a M\"obius band.
On the other hand, collapsing the direction of $v_2$, the limiting orbifold is the quotient of $\R^2$ by the group of isometries generated by $T_1(x,y)=(x+a,y)$, $T_2(x,y)=\big(x+\frac12a,y+b\big)$, and $T_3(x,y)=(x+\frac12a,-y)$. Here, $a=\vert v_1\vert>0$ and $b=\vert v_3\vert>0$. Clearly, $T_3^2=T_1$, so that $T_1$ can be omitted from the list of generators. This quotient can be identified with a Klein bottle. Collapsing (generic) directions with irrational slope with respect to $v_1$ and $v_2$ has the same effect as collapsing all of $W_1$, which results in an interval as the limit.
\end{example}

\begin{example}[Case $\mathcal G_2$]\label{exa:Z2}
The basis of the lattice $L_\pi$ of the Bieberbach group $\pi=\mathcal G_2$ is such that $v_1$ is orthogonal to $v_2$ and $v_3$.
This group $\pi$ is generated by $L_\pi$ and $(A,\frac12v_1)$, where $A(v_1)=v_1$, and $A(v)=-v$ for all $v$ in the span of $v_2$ and $v_3$.
Thus, $H_\pi\cong\Z_2$ and it has one trivial isotypic component $W_1=\operatorname{span}\{v_1\}$ and one nontrivial isotypic component $W_2=\operatorname{span}\{v_2,v_3\}$, isomorphic to the direct sum of two copies of the sign representation of $\Z_2$. 
Collapsing $W_1$, the limit is the quotient of a $\R^2$ by the group of isometries generated by $T_1$, $T_2$, $T_3$ and $T_4=-\mathrm{Id}$, with $T_1(x,y)=(x+a,y)$, $T_2(x,y)=(x,y+b)$, and $T_3(x,y)=(x+\frac a2,y)$. Here, $a=\vert v_1\vert$ and $b=\vert v\vert$. Clearly, $T_1=T_3^2$. A fundamental domain for this quotient is given by the rectangular triangle with vertices $(0,0)$, $(a,0)$ and $(0,b)$. This limit can be identified with the flat $2$-orbifold $D^2(2,2;)$.
Collapsing $W_2$, the limit is clearly $S^1$. Collapsing a one-dimensional subspace of $W_2$ produces as limit either a flat Klein bottle or a point, depending on the slope of the subspace.
\end{example}

\appendix
\section{\texorpdfstring{Flat $2$-orbifolds}{Flat two-dimensional orbifolds}}\label{sec:appflat2orbs}
We give here some details on how to recognize the flat $2$-orbifolds resulting from collapse of flat $3$-manifolds discussed in Section~\ref{sec:examples}.
In all the examples below, $a$ and $b$ are positive constants. In all figures, singularities (which are always interior points) are marked in red. Red lines represent the boundary. Arrows indicate identifications of the corresponding sides.

\subsection{Groups of isometries}
In Example~\ref{exa:Z2Z2}, case (a), we consider the group of isometries generated by:
\begin{equation}\label{eq:exa7}
\begin{aligned}
&T_1(x,y)=(x+a,y),& &T_2(x,y)=(x,y+b), \\
&T_3(x,y)=(-x,-y),& &T_4(x,y)=\big(x+\tfrac12a,-y+\tfrac12b\big),\\
&T_5(x,y)=\big(-x+\tfrac12a,y+\tfrac12b\big).&
\end{aligned}
\end{equation} 
Clearly, $T_5=T_4\circ T_3$, so this generator is redundant, and it can be removed from the list. 
Denote by $G_1$ the group of isometries of $\R^2$ generated by $T_1$, $T_2$, $T_3$ and $T_4$ in \eqref{eq:exa7}. It is not hard to check that cases (b) and (c) of Example~\ref{exa:Z2Z2} are essentially equal to case (a) by a change of coordinates.

\begin{figure}[ht]
\includegraphics[scale=.3]{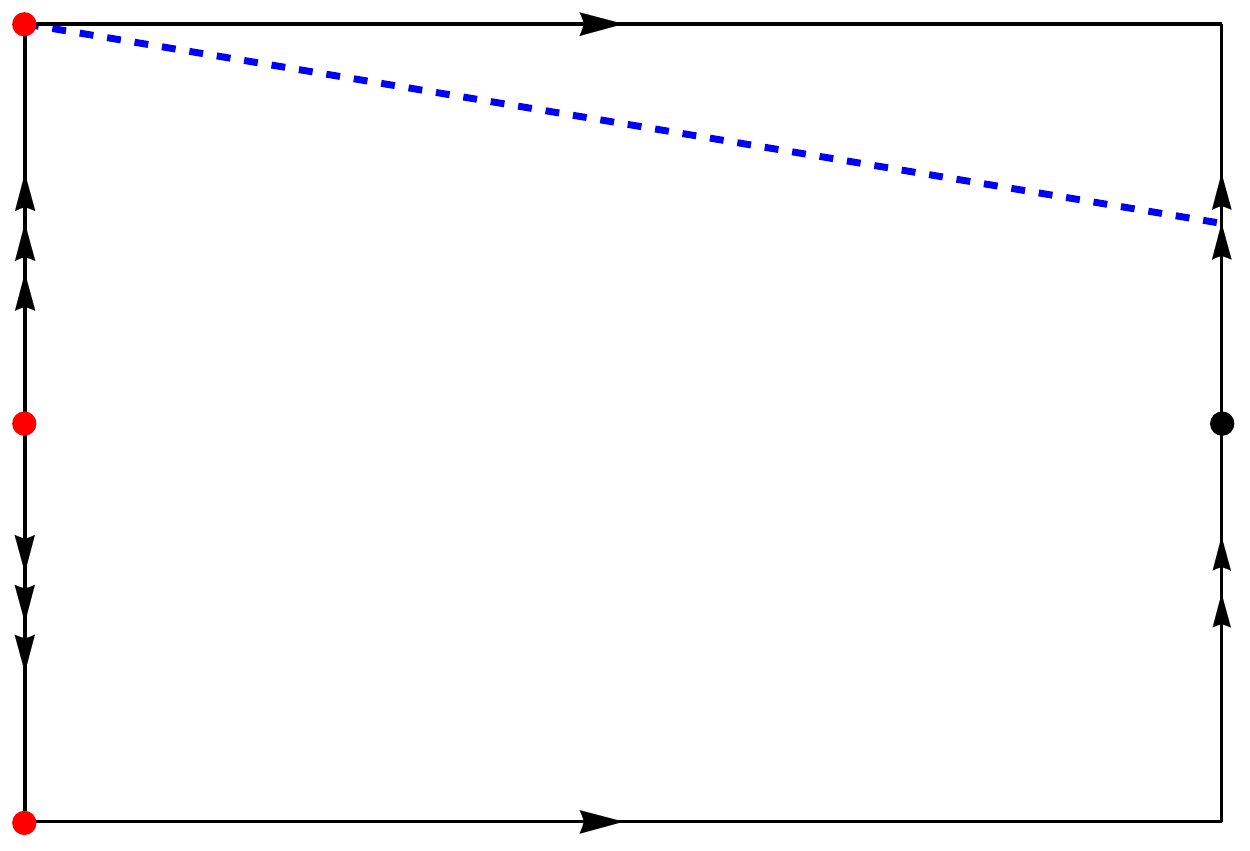}
\caption{This picture represents the rectangle $[0,a/4]\times[-b/2,b/2]$, which is a fundamental domain for the action of $G_1$. The arrows give the identifications of the sides that produce the quotient $\R^2/G_1$, as described in Proposition~\ref{quott}.}
\label{fig:caseG1-1}
\includegraphics[scale=.3]{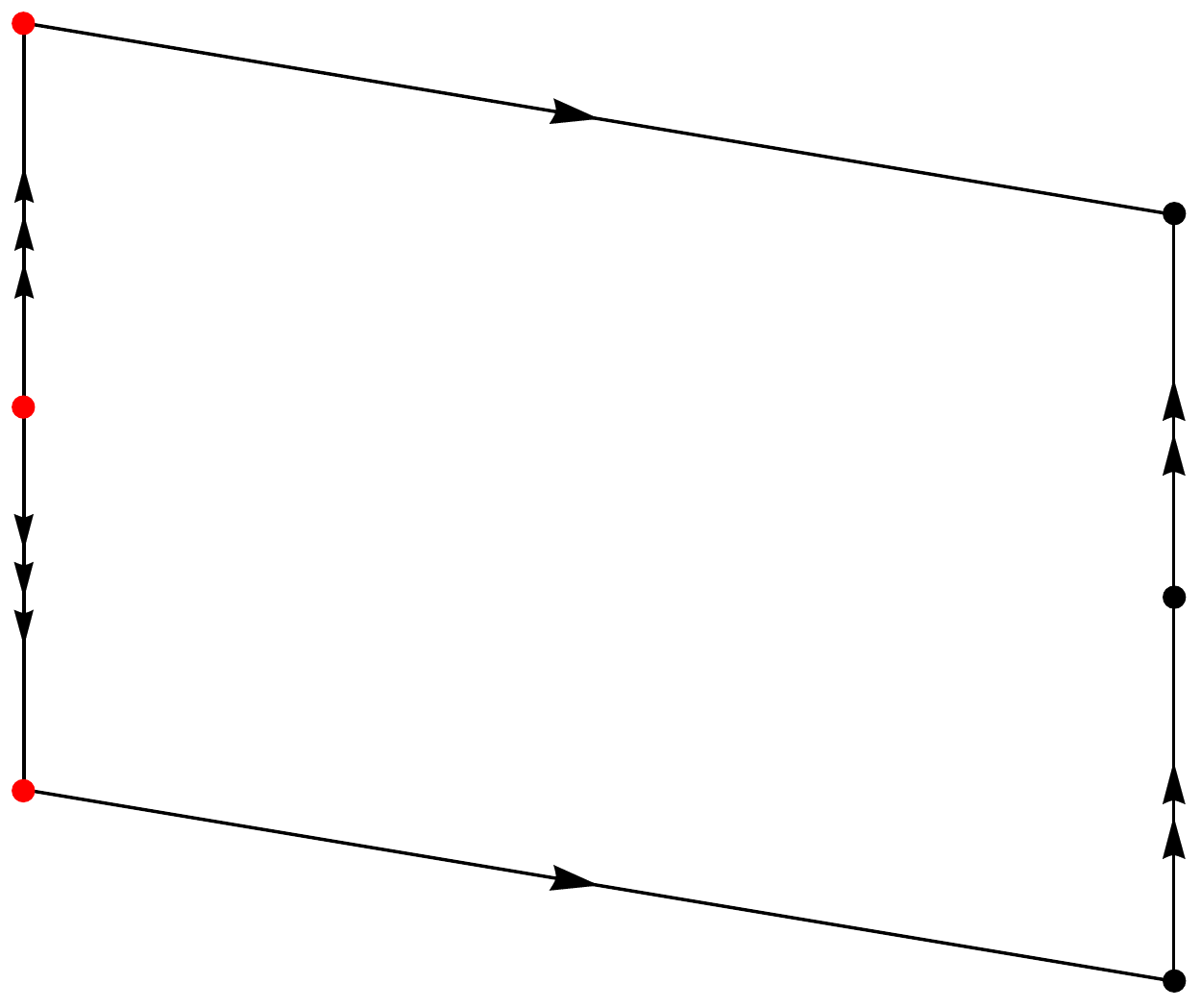}\qquad\qquad
\includegraphics[scale=.25]{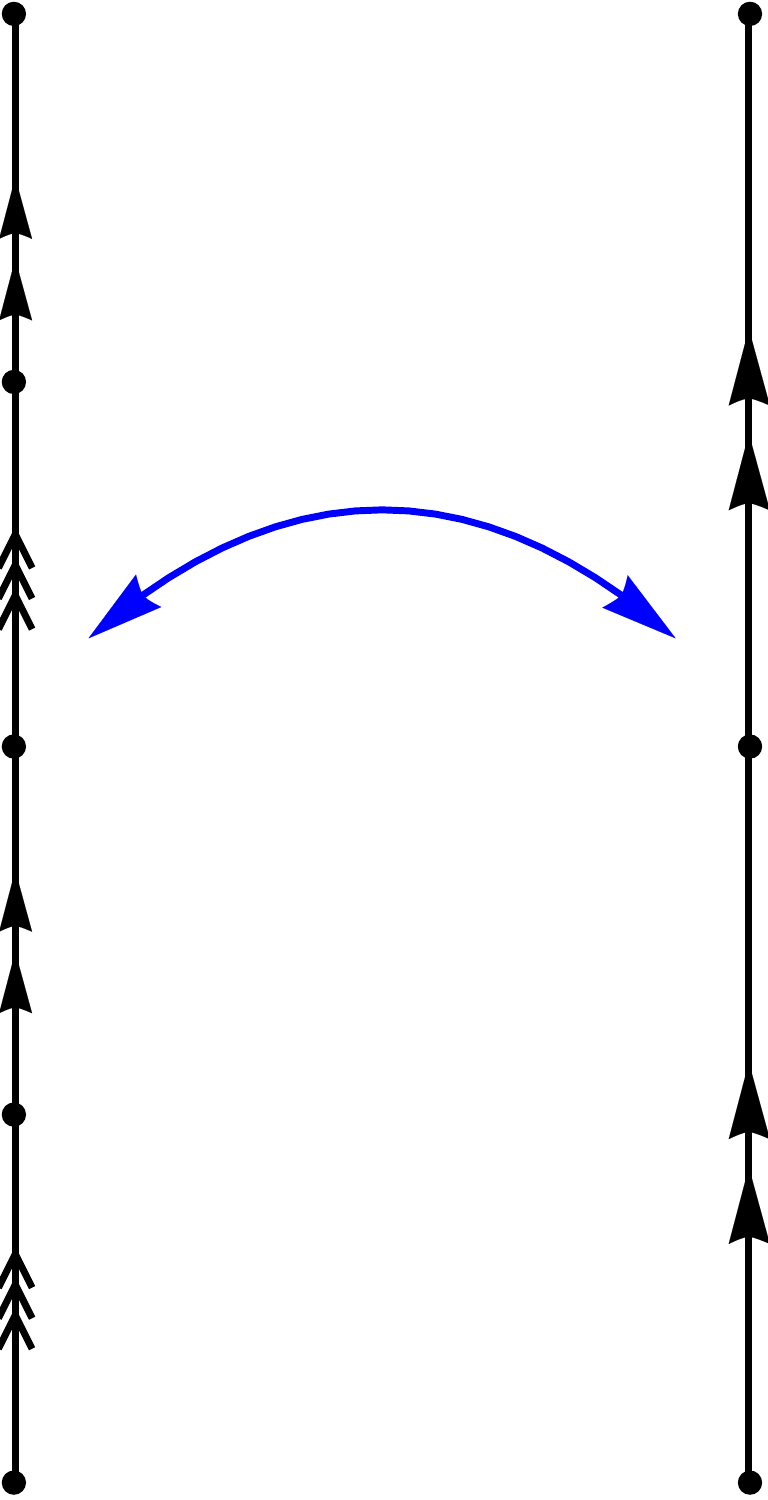}
\caption{From Figure~\ref{fig:caseG1-1}, translate the triangle in the upper right corner downward by $(0,-b)$. By relabeling the identifications according to the figure on the right, the quotient $\R^2/G_1$ can be identified with $\R P^2(2,2;)$.}
\label{fig:caseG1-2}
\end{figure}

In Example~\ref{exa:Z2Z2no1}, case (a), we consider the group of isometries generated by:
\begin{equation}\label{eq:exa1}
\begin{aligned}
&T_1(x,y)=(x+a,y),& &T_2(x,y)=(x,y+b),\\
&T_3(x,y)=(-x,-y),& &T_4(x,y)=\big(x+\tfrac a2,-y\big).
\end{aligned}
\end{equation}
Denote by $G_2$ the group of isometries of $\R^2$ generated by $T_1$, $T_2$, $T_3$ and $T_4$ in \eqref{eq:exa1}.

\begin{figure}[ht]
\includegraphics[scale=.3]{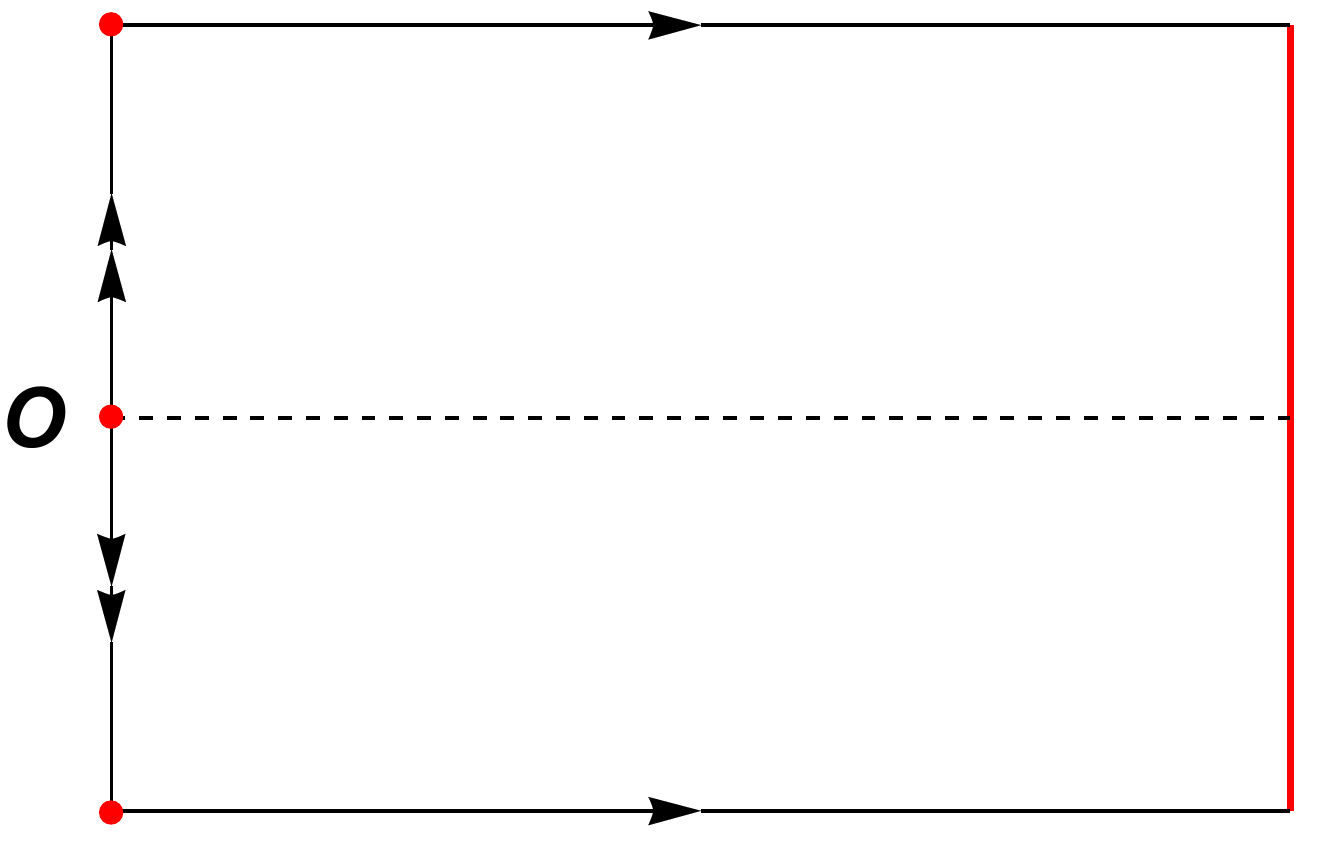}
\caption{The rectangle $[0,a/4]\times[-b/2,b/2]$ is a fundamental domain for the action of~$G_2$. The identifications of the sides are described in Proposition~\ref{quott}.}
\label{fig:caseG2-1}
\includegraphics[scale=.3]{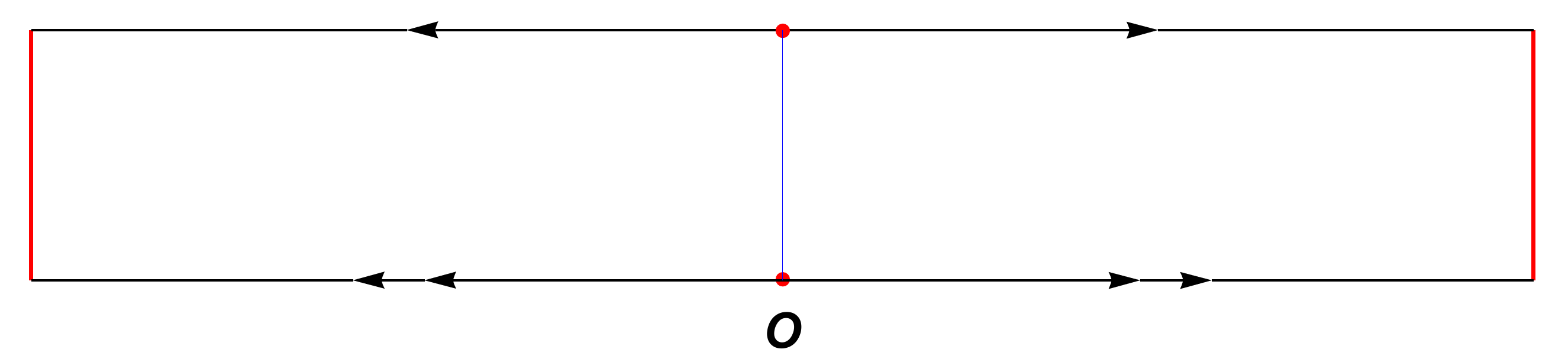}
\caption{The result of rotating clockwise the lower half of the rectangle by the angle $\pi$ around the origin $O$, in Figure~\ref{fig:caseG2-1}.}
\label{fig:caseG2-2}
\includegraphics[scale=.3]{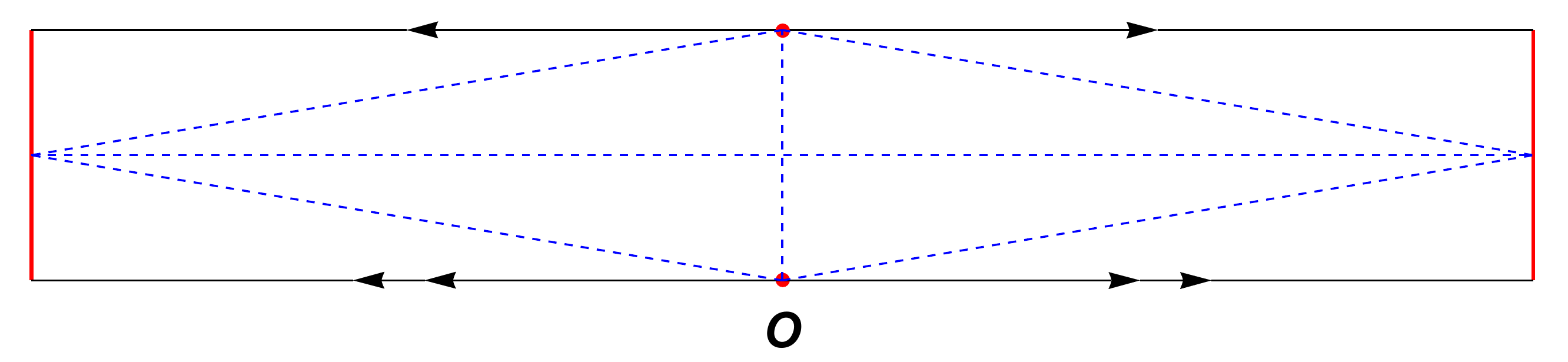}
\caption{Triangulation of the rectangle in Figure~\ref{fig:caseG2-2}. The quotient $\R^2/G_2$ is identified with $D^2(2,2;)$.}
\label{fig:caseG2-3}
\end{figure}

\begin{figure}[ht]
\includegraphics[scale=.3]{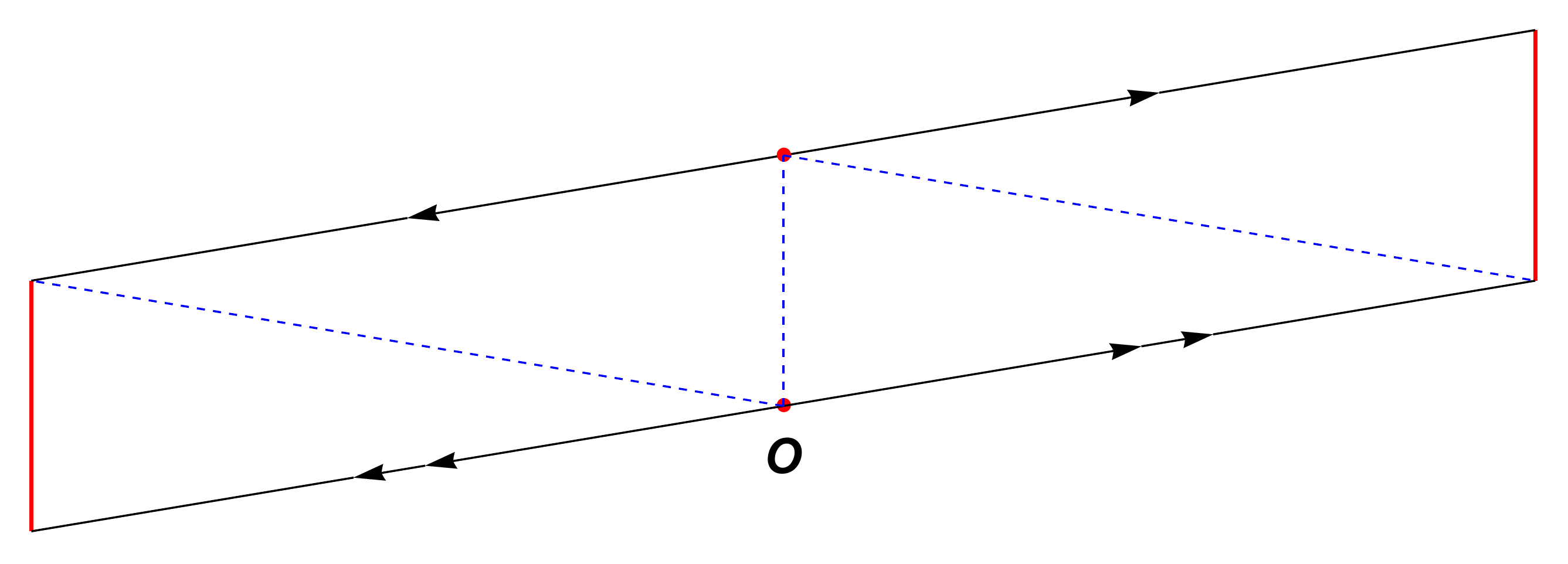}
\caption{This figure also represents the quotient $\R^2/G_2$. It is obtained from Figure~\ref{fig:caseG2-3} by:
(a) rotating the triangle in the lower right corner clockwise by the angle $\pi$ around $O$;
(b) rotating the triangle in the upper left corner clockwise by the angle $\pi$ around the other red point;
(c) re-triangulating into $4$ congruent isosceles triangles.}
\label{fig:caseG2-4}
\end{figure}

In Example~\ref{exa:Z2Z2no1}, case (b), we consider the group of isometries generated\footnote{%
To simplify notation, here we are switching the names of $x$ and $y$, and of $a$ and $b$.} by:
\begin{equation}\label{eq:exa2}
\begin{aligned}
&T_1(x,y)=(x+a,y),& &T_2(x,y)=(x,y+b),\\
&T_3(x,y)=(-x,y),& &T_4(x,y)=\big(-x,y+\tfrac b2\big).
\end{aligned}
\end{equation}
Denote by $G_3$ the group of isometries of $\R^2$ generated by $T_1$, $T_2$, $T_3$ and $T_4$ in \eqref{eq:exa2}.

\begin{figure}[ht]
\includegraphics[scale=.25]{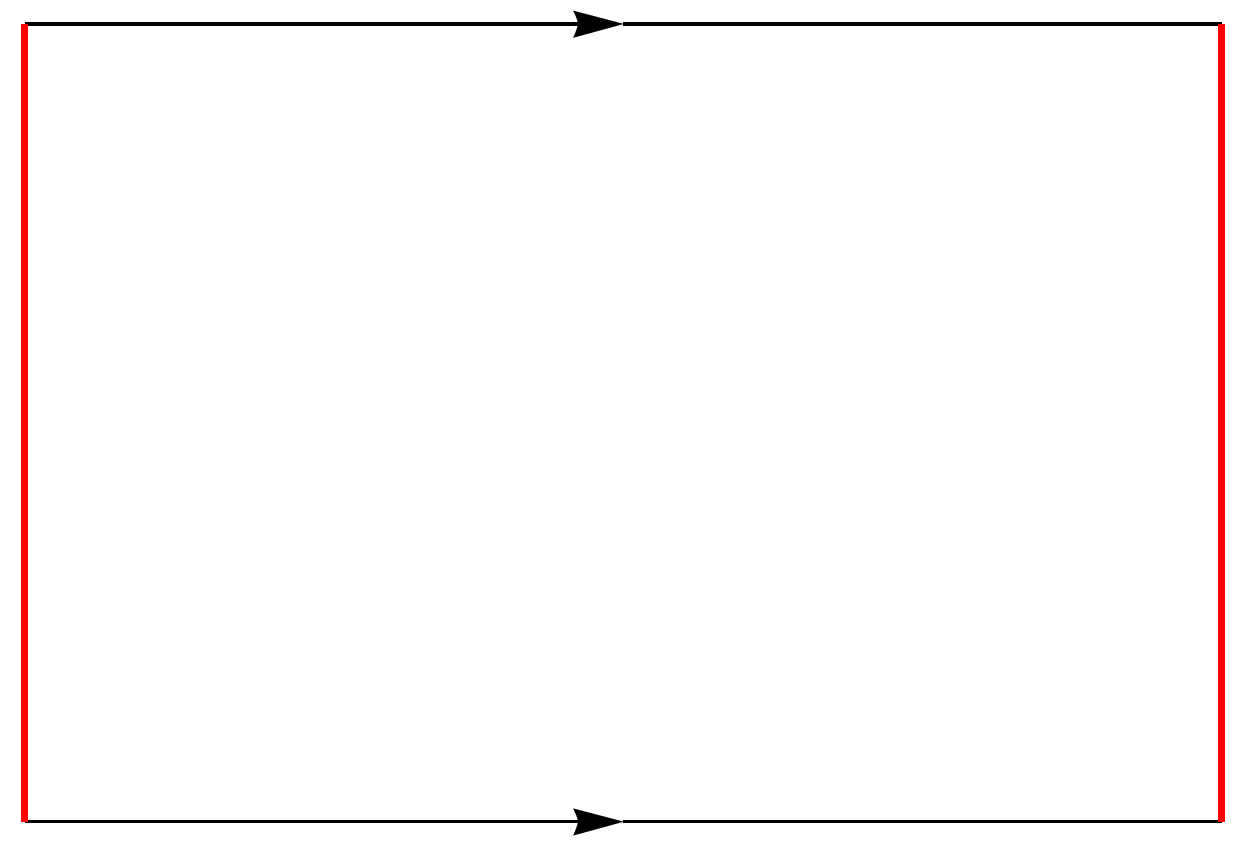}\qquad
\includegraphics[scale=.25]{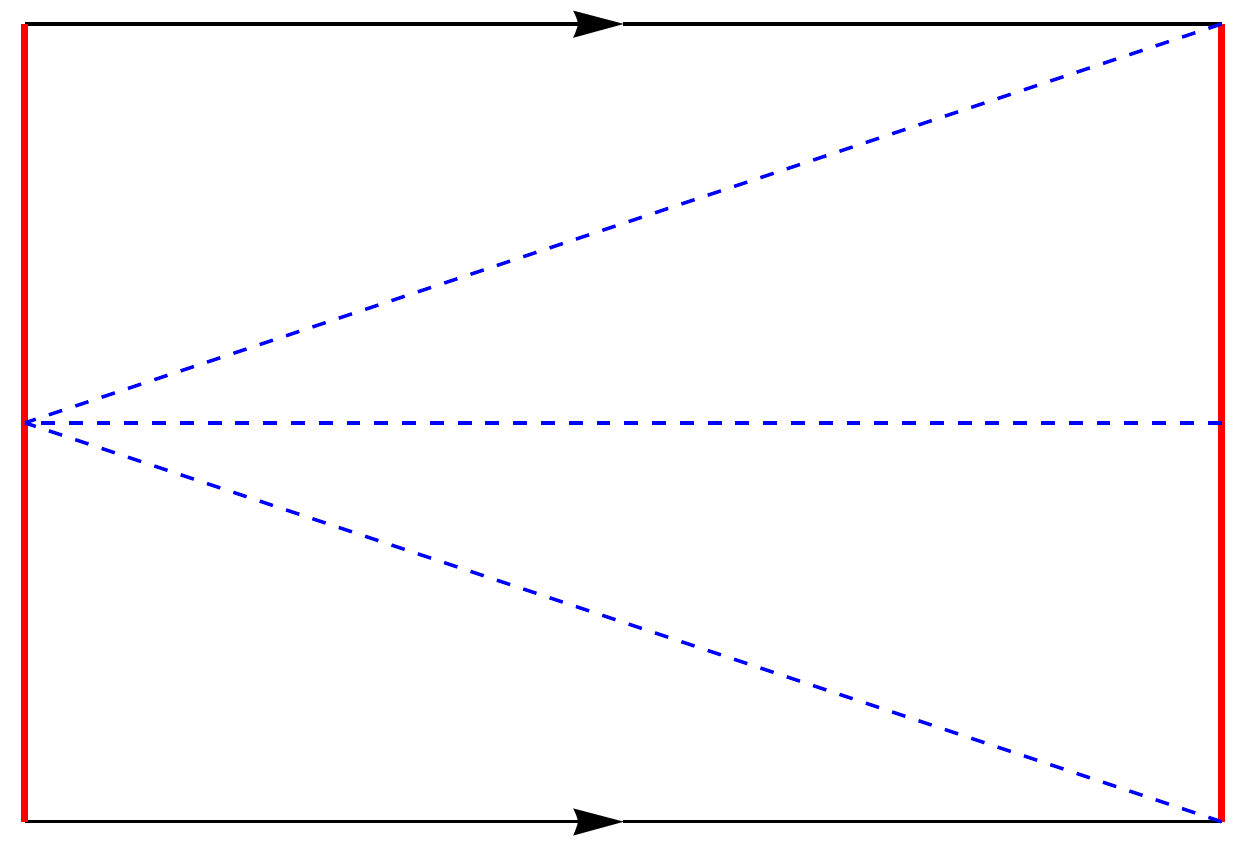}
\caption{The rectangle $[a/2,b/2]$ is a fundamental domain for the action of $G_3$ on $\R^2$. On the right, a triangulation of the quotient space $\R^2/G_3$, which is identified with $S^1\times I$.}
\label{fig:caseG3-1}
\end{figure}

In Example~\ref{exa:Z2Z2no2}, case (a), we consider the group of isometries generated by:
\begin{equation}\label{eq:exa4}
\begin{aligned}
&T_1(x,y)=(x+a,y),& &T_2(x,y)=(x,y+b),\\
&T_3(x,y)=(-x,-y),& &T_4(x,y)=\big(x+\tfrac a2,y+\tfrac b2\big).
\end{aligned}
\end{equation}
Denote by $G_4$ the group of isometries of $\R^2$ generated by $T_1$, $T_2$, $T_3$ and $T_4$ in \eqref{eq:exa4}.

\begin{figure}
\includegraphics[scale=.3]{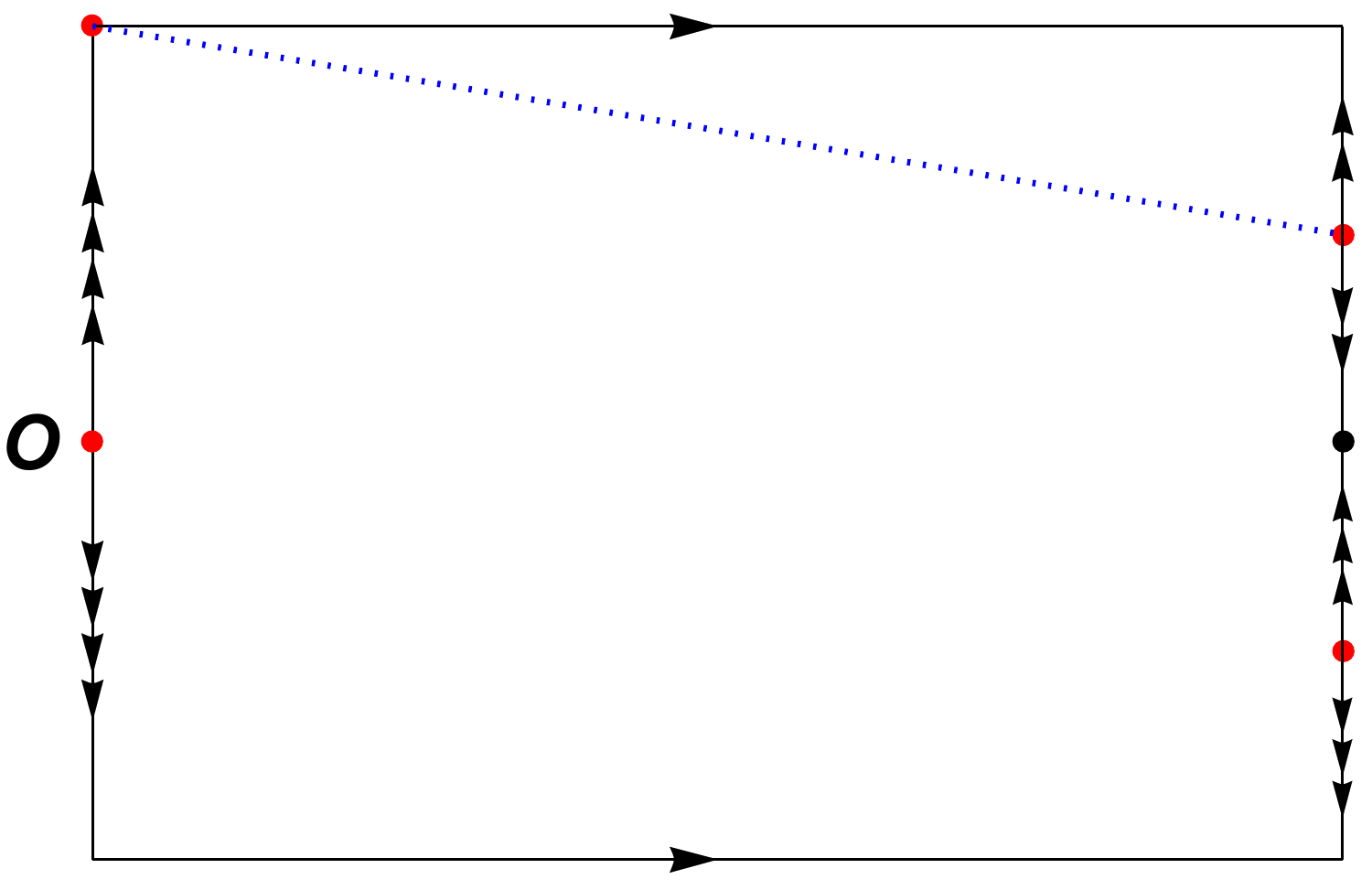}
\caption{The rectangle $[0,a/4]\times[-b/2,b/2]$ is a fundamental domain for the action of $G_4$.}
\label{fig:caseG4-1}
\includegraphics[scale=.3]{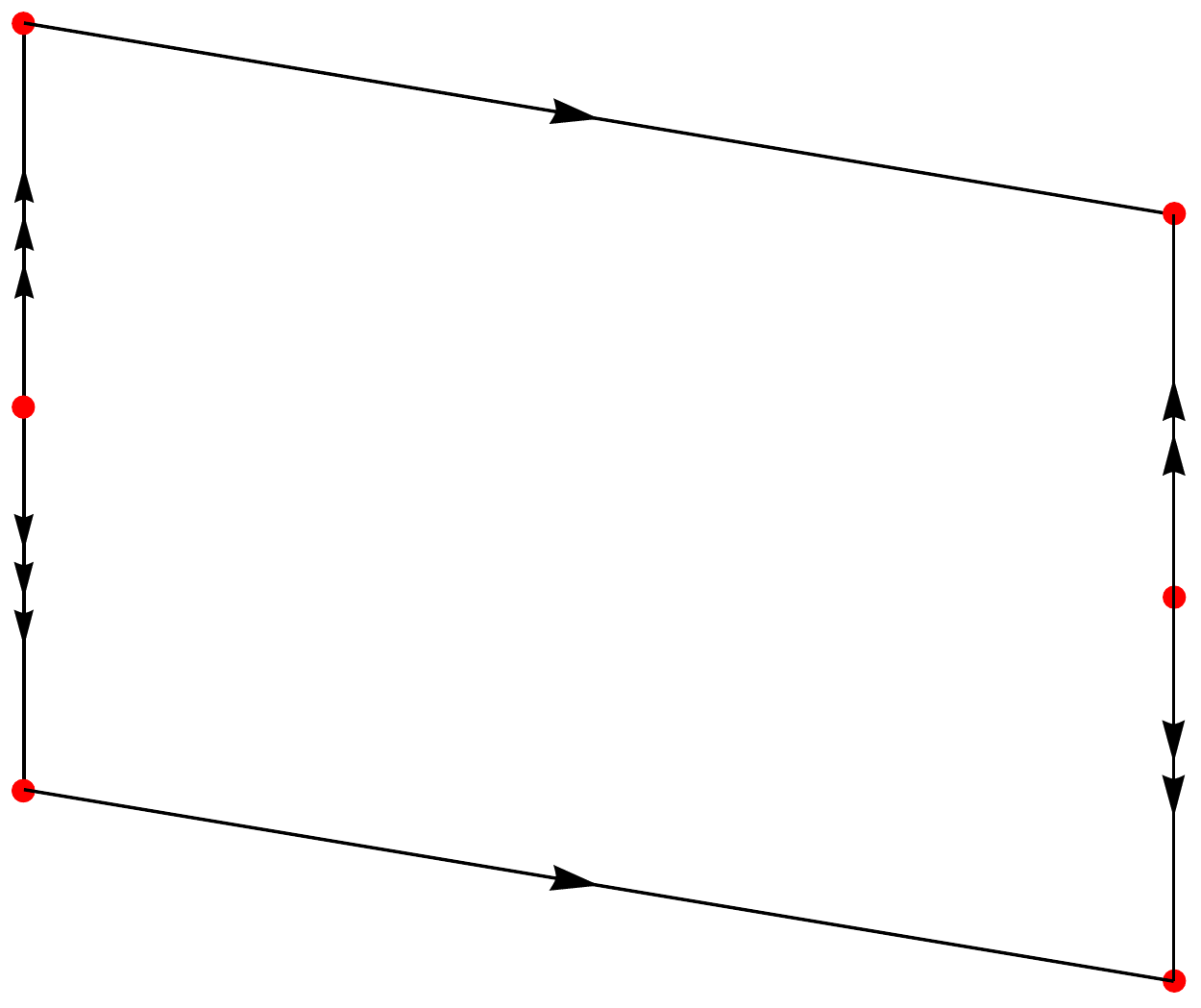}\qquad\qquad
\includegraphics[scale=.3]{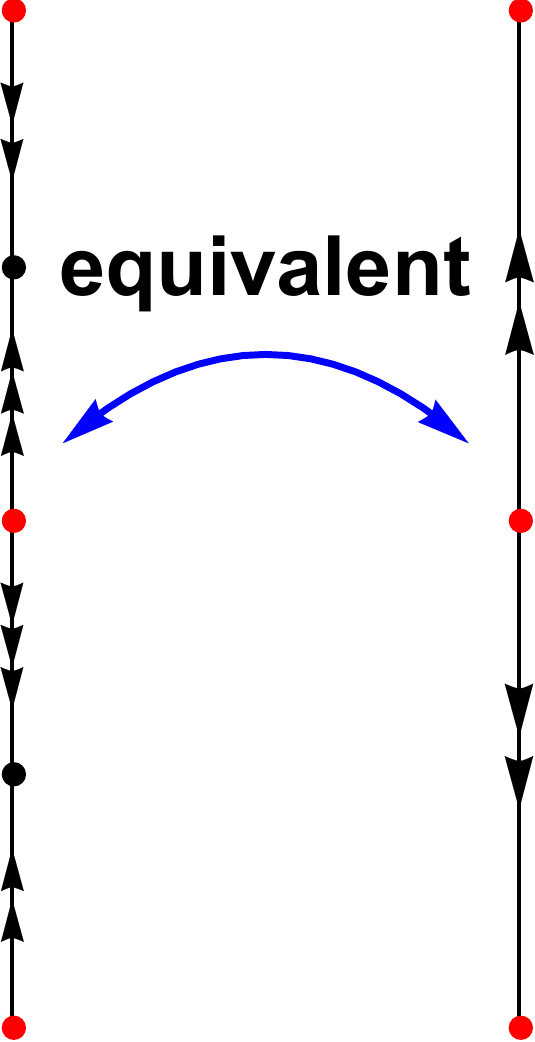}
\caption{Translate the triangle in the upper right corner downward by $(0,-b)$. By relabeling the identifications according to the figure on the right, the quotient $\R^2/G_4$ can be easily identified with $S^2(2,2,2,2;)$.}\label{fig:caseG4-3}
\end{figure}

In Example~\ref{exa:Z2Z2no2}, case (b), we consider the group of isometries generated by
$T_1(x,y)=(x+a,y)$, $T_2(x,y)=(x,y+b)$, $T_3(x,y)=(x,-y+\tfrac12b)$ and $T_4(x,y)=(x+\frac12a,-y)$. 
Using new coordinates $(x,y-\tfrac14b)$ and interchanging the names of $x$ and $y$, and $a$ and $b$, it is easily seen that the quotient is isometric to the quotient of $\R^2$ by the group of isometries generated by:
\begin{equation}\label{eq:exa5}
\begin{aligned}
&T_1(x,y)=(x+a,y),& &T_2(x,y)=(x,y+b),\\
&T_3(x,y)=(-x,y),& &T_4(x,y)=\big(-x+\tfrac a2,y+\tfrac b2\big).
\end{aligned}
\end{equation}
Denote by $G_5$ the group of isometries of $\R^2$ generated by $T_1$, $T_2$, $T_3$ and $T_4$ in \eqref{eq:exa5}.

\begin{figure}
\includegraphics[scale=.3]{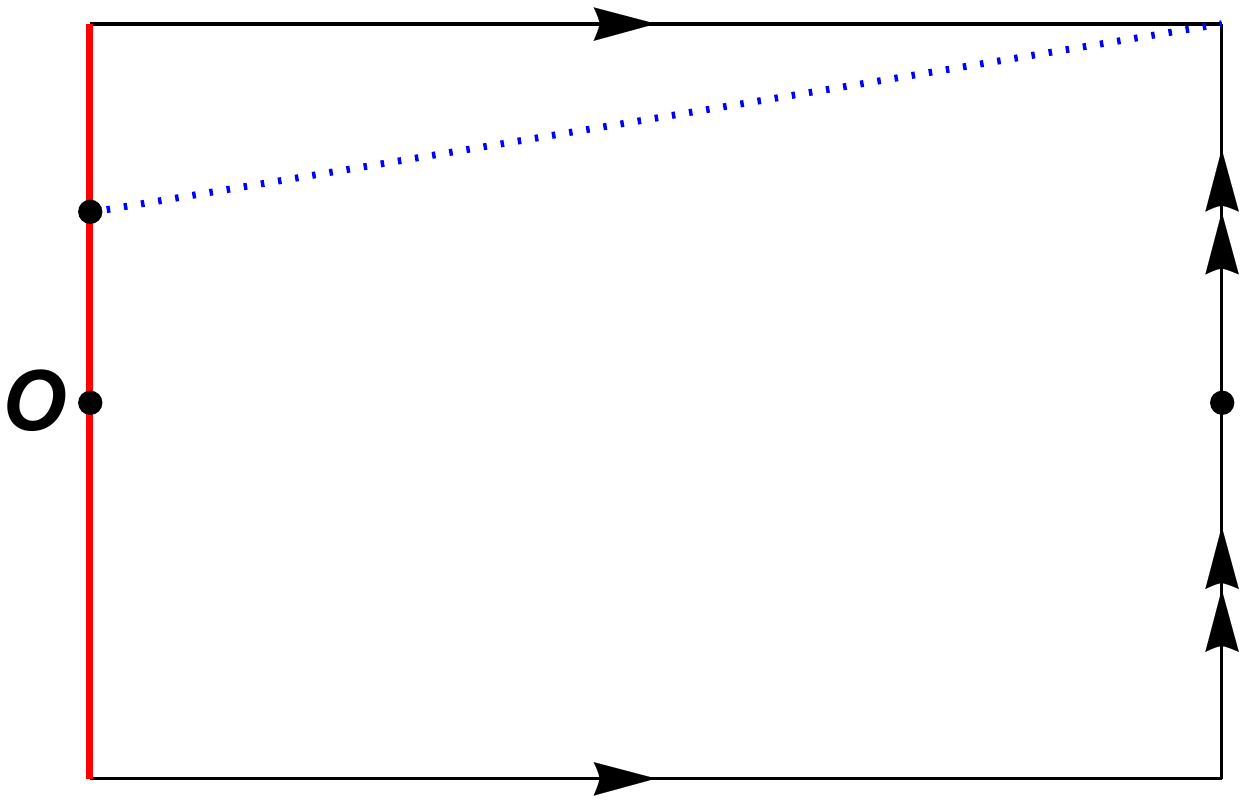}
\caption{The rectangle $[0,a/4]\times[-b/2,b/2]$ is a fundamental domain for the action of $G_5$.}
\label{fig:caseG5-1}
\includegraphics[scale=.3]{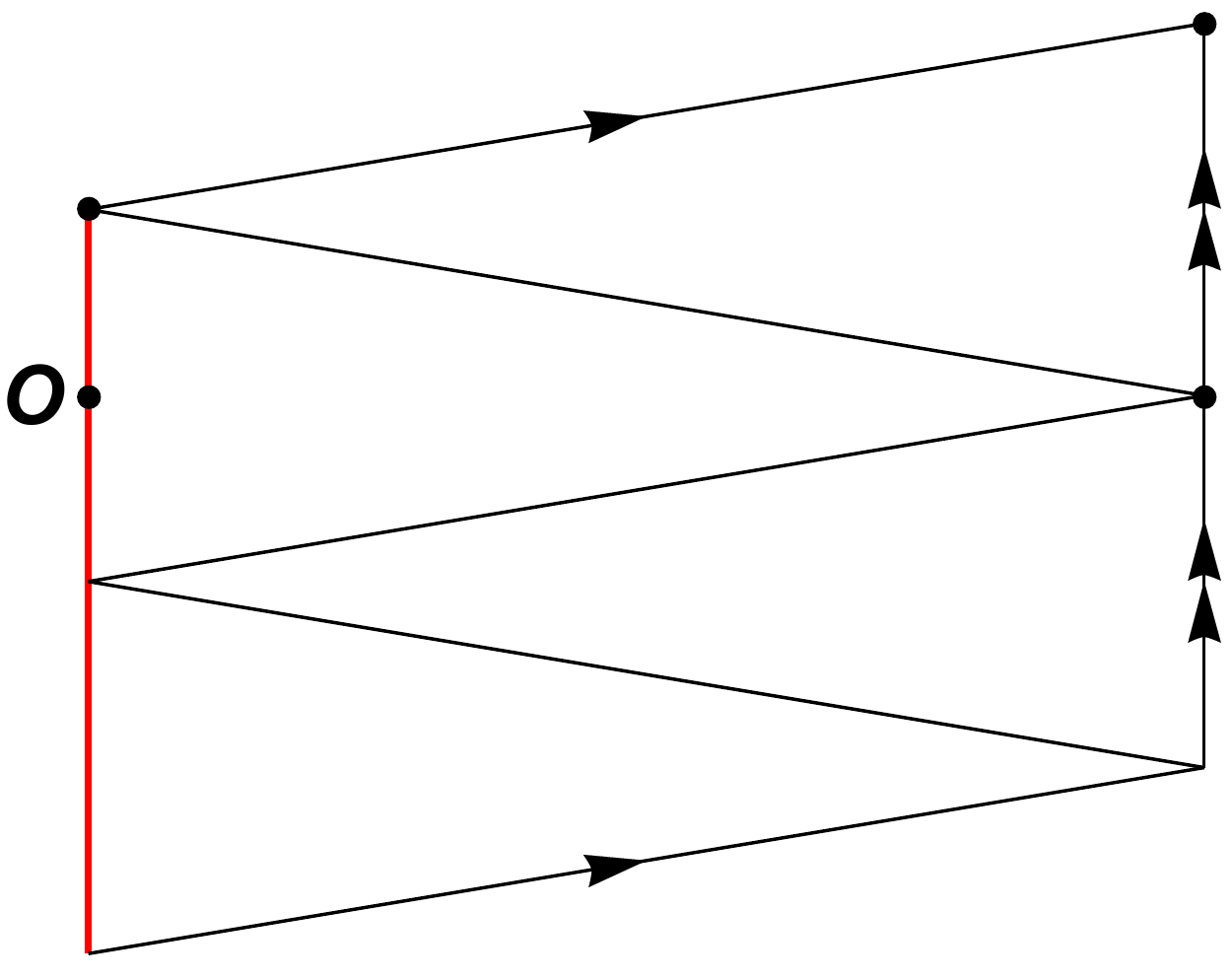}\qquad\qquad
\caption{From Figure~\ref{fig:caseG5-1}, translate the triangle in the upper left corner downward by $(0,-b)$. The quotient $\R^2/G_5$ can be identified with a flat M\"obius band without  singularities.}
\label{fig:caseG5-2}
\end{figure}

In Example~\ref{exa:Z2no}, case (a), we consider the group of isometries generated by $(x,y)\mapsto (x+a,y)$, $(x,y)\mapsto\big(x+\frac12a,y+b\big)$, and $(x,y)\mapsto(x,-y)$.
Switching the name of the variables (we let the new symbols $x$, $y$, $a$ and $b$ stand for the old symbols $y$, $x$, $2b$ and $a$), a set of generators for this group is given by:
\begin{equation}\label{eq:exa10}
T_2(x,y)=(x,y+b),\quad T_3(x,y)=(-x,y),\quad T_4(x,y)=\big(x+\tfrac a2,y+\tfrac b2\big).
\end{equation}
Denote by $G_6$ the group of isometries of $\R^2$ generated by $T_2$, $T_3$ and $T_4$ in \eqref{eq:exa10}.
\smallskip

\begin{figure}
\includegraphics[scale=.275]{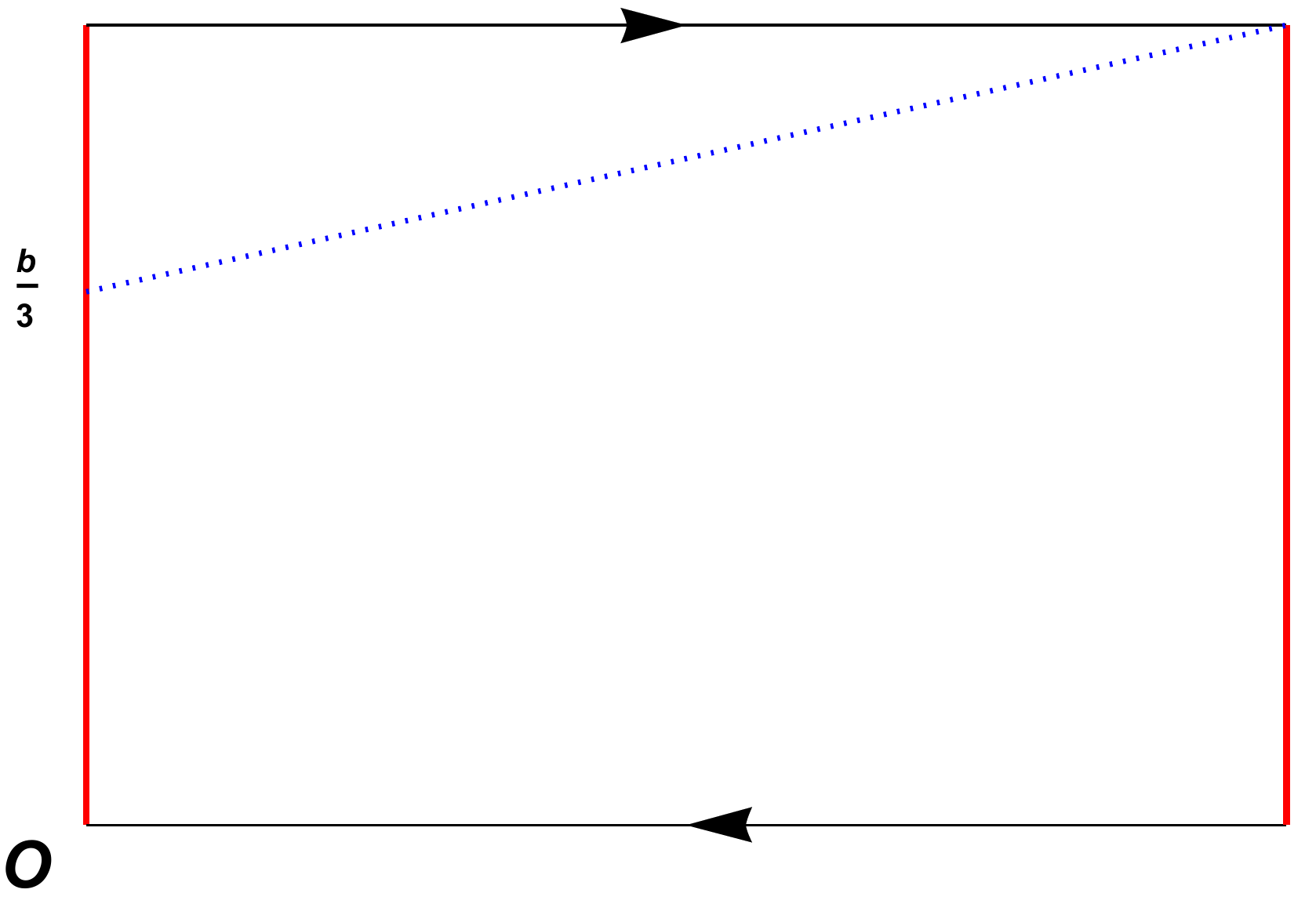}
\caption{The rectangle $[0,a/2]\times [0,b/2]$ is a fundamental domain for the action of $G_6$ on $\R^2$. The identifications of the sides are described in Proposition~\ref{quott}.}
\label{fig:caseG6-1}
\includegraphics[scale=.25]{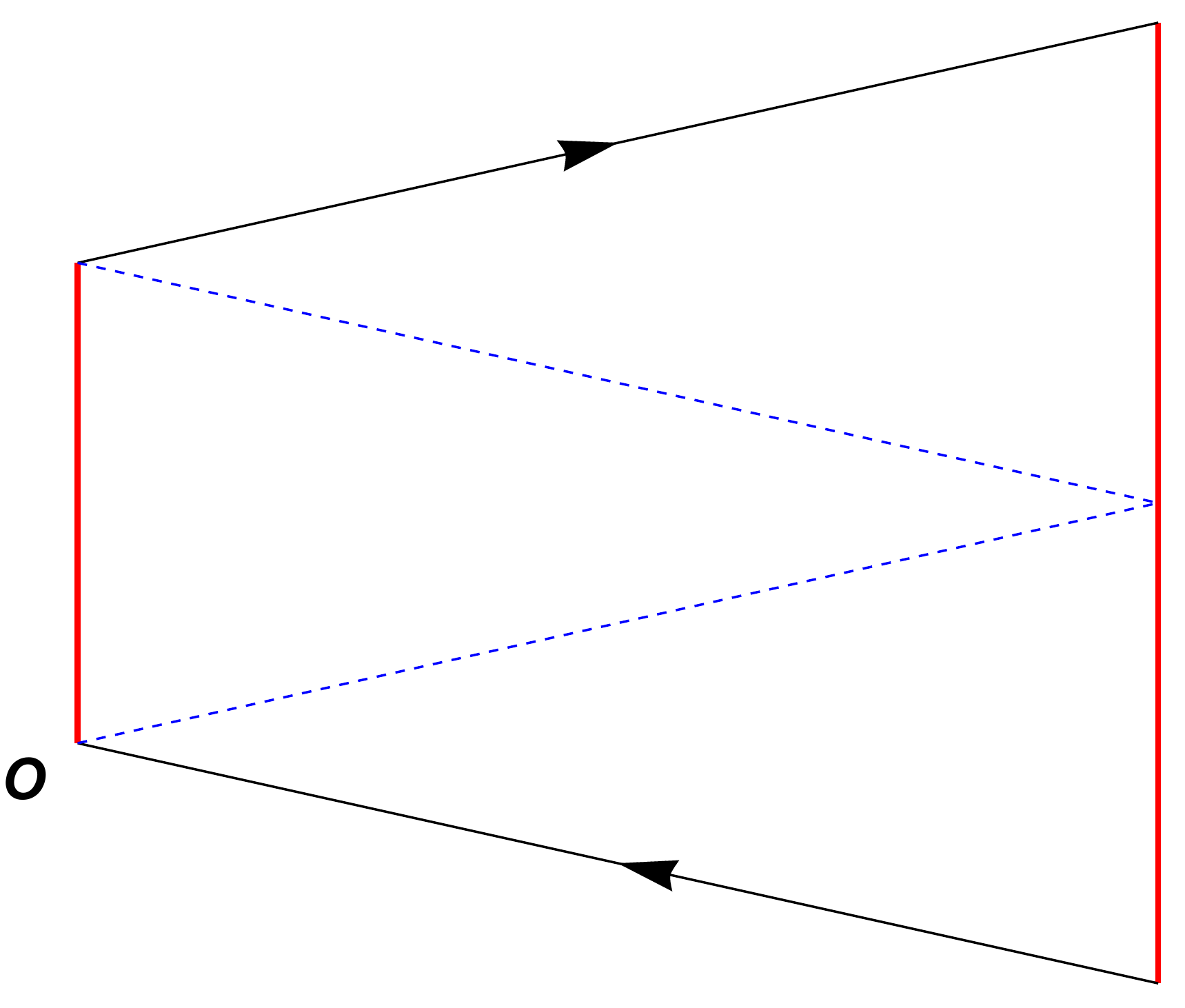}
\caption{In Figure~\ref{fig:caseG6-1}, translate the triangle in the upper left corner downward by $(0,-b/2)$, and then reflect the resulting image triangle about the vertical line through the center of the rectangle. This quotient space is a flat M\"obius band without singularities.}
\label{fig:caseG6-2}
\end{figure}

In Example~\ref{exa:Z2no}, case (b), we consider the group of isometries generated by $(x,y)\mapsto(x+a,y)$, $(x,y)\mapsto\big(x+\frac12a,y+b\big)$, and $(x,y)\mapsto(x+\frac12a,-y)$.
Changing the name of the variables (we let the new symbols $x$, $y$, $a$ and $b$ stand for the old symbols $y$, $x$, $2b$ and $a$), a set of generators for this group is given by:
\begin{equation}\label{eq:exa11}
T_2(x,y)=(x,y+b),\quad T_3(x,y)=\big(\!-x,y+\tfrac b2\big),\quad T_4(x,y)=\big(x+\tfrac a2,y+\tfrac b2\big).
\end{equation}
Denote by $G_7$ the group of isometries of $\R^2$ generated by  $T_2$, $T_3$ and $T_4$ in \eqref{eq:exa11}.

\begin{figure}
\includegraphics[scale=.3]{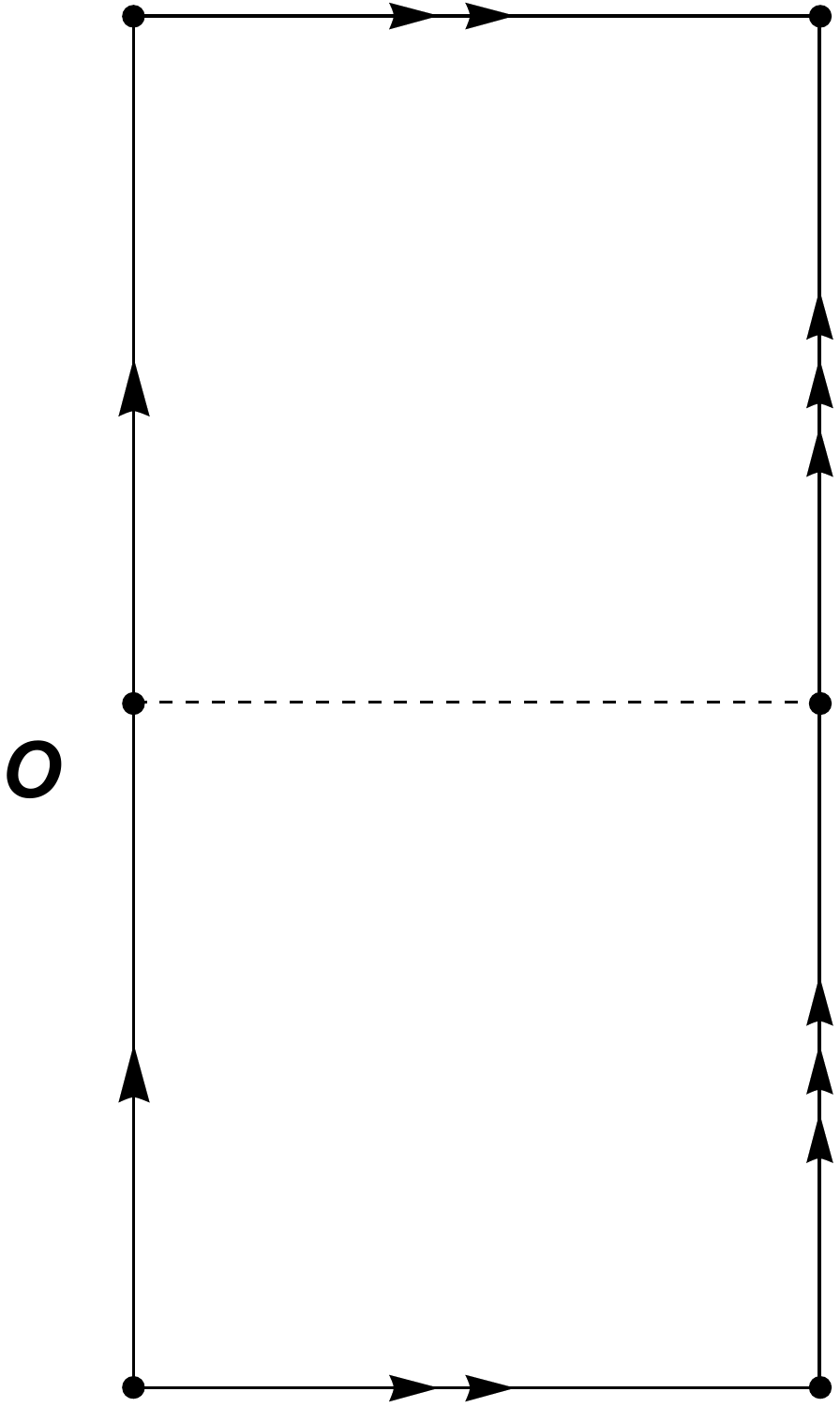}
\caption{The rectangle $[0,a/4]\times[-b/2,b/2]$ is a fundamental domain for the action of $G_7$.}
\label{fig:caseG7-1}
\includegraphics[scale=.3]{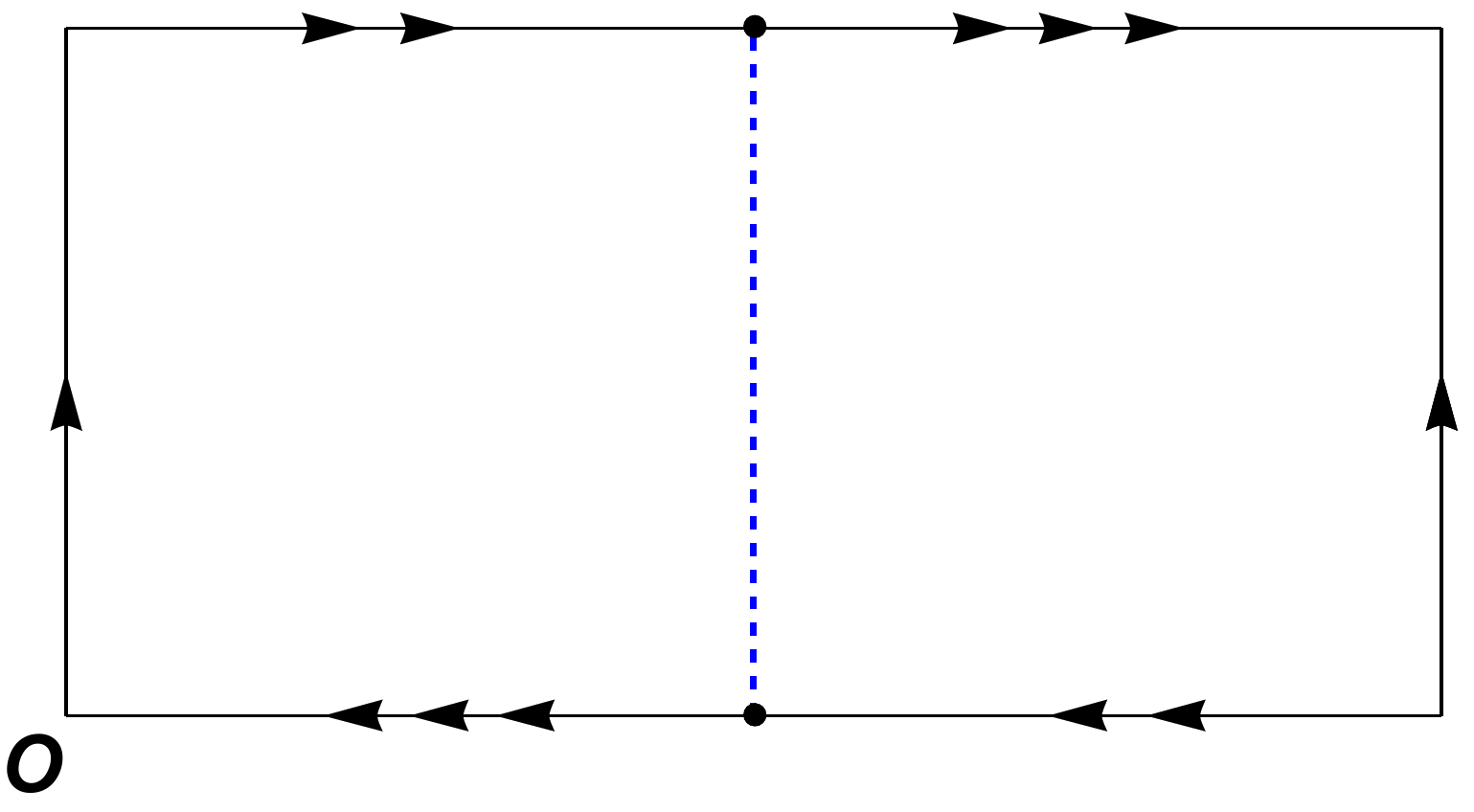}
\caption{In Figure ~\ref{fig:caseG7-1}, reflect the lower half of the rectangle about its right vertical side, and then translate the resulting image half-rectangle upward by $(0,b/2)$.}
\label{fig:caseG7-2}
\includegraphics[scale=.3]{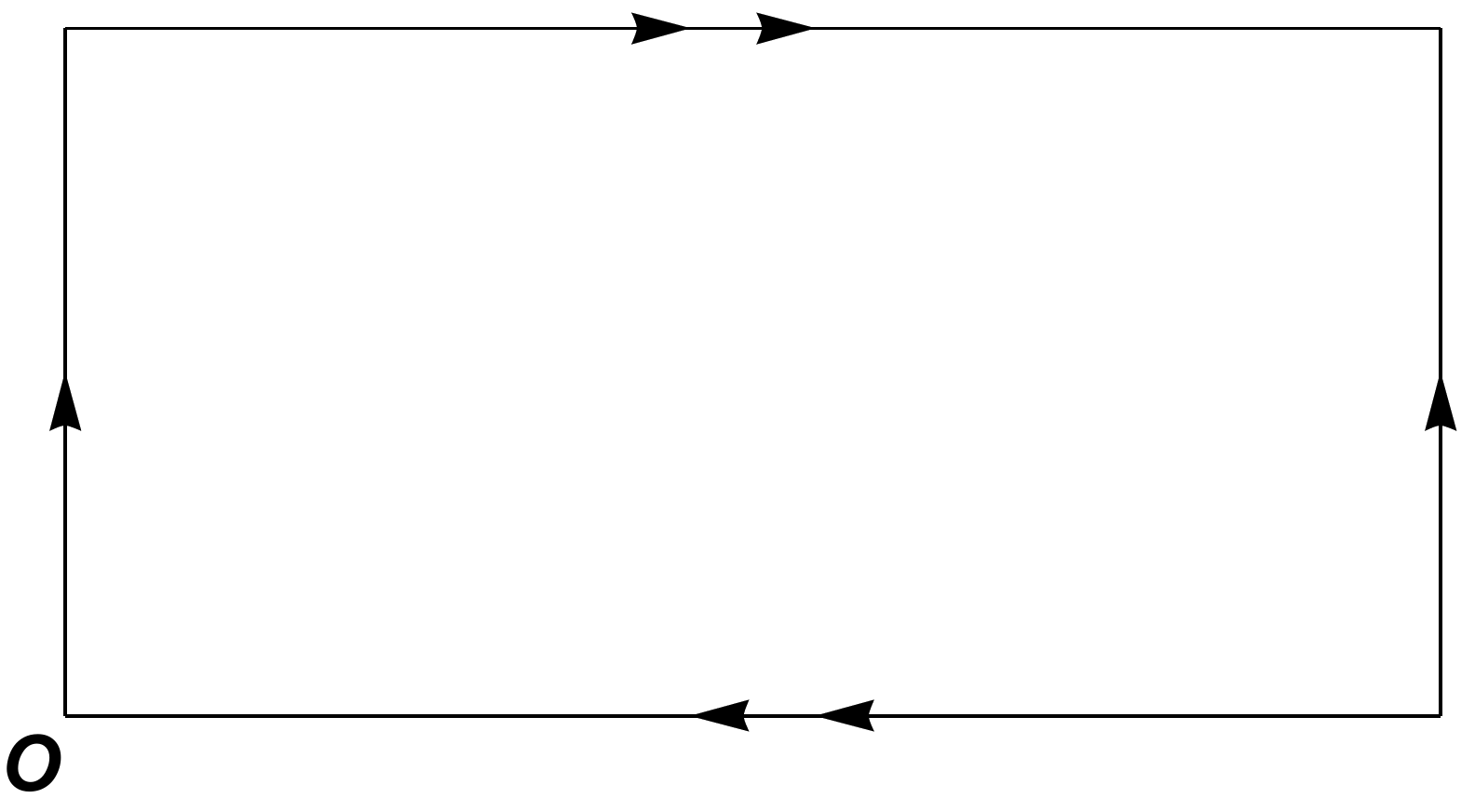}
\caption{Figure~\ref{fig:caseG7-2} depicts the standard flat Klein bottle.}
\label{fig:caseG7-3}
\end{figure}

We refer to all examples above collectively letting $G$ denote the discrete group of isometries of $\R^2$.

\begin{remark}\label{lpart}
In each example, the image $H$ of $G$ under the lin\-e\-ar-part homo\-mor\-phism is contained in the matrix 
group $\,\{\mathrm{diag}\hs(\delta,\nh\ve):\delta,\ve\in\{1,\nnh-\nnh1\}\}\cong\Z_2\w\oplus\mathds Z_2\w$. For $G=G_3$, $G=G_5$, $G=G_6$ or $G=G_7$,
$\,H=\{\mathrm{diag}\hs(1,\nnh1),\mathrm{diag}\hs(1,\nnh-\nnh1)\}$. On the 
other hand, 
$\,H=\{\mathrm{diag}\hs(1,\nnh1),\mathrm{diag}\hs(\nnh-\nnh1,\nnh-\nnh1)\}\,$ 
for $G=G_4$.
\end{remark}

\subsection{Fundamental domains and quotient spaces}
We now indicate how to identify fundamental domains for each of the above group actions on $\R^2$ and recognize the quotient space using boundary identifications on such fundamental domains.

\begin{lemma}\label{onefs}
Every $G$-or\-bit intersects the rectangle
$Q=[\hs0,a/2\hh]\times[\hs0,b/2\hh]$ for $G=G_3, G_6$, and $Q=[\hs0,a/4\hh]\times[-b/2,b/2\hh]$ for $G=G_1, G_2, G_4,G_5, G_7$.
\end{lemma}

\begin{proof}
If a group $\,G\,$ of af\-fine transformations of a real vector 
space contains the translations by all elements of a basis 
$\,e_1\w,\dots,e_n\w$, then every orbit of $\,G\,$ intersects the set 
$\,\{\mu_1\w e_1\w\nh+\ldots+\mu_n\w e_n\w:
\mu_1\w,\dots,\mu_n\w\nh\in[\nh-\nh1/2,1/2\hh]\}$. 
In fact, any $\,x\in\R\,$ differs by an integer from some 
$\,y\in[\nh-\nh1/2,1/2\hh]$.

All $\,G$-or\-bits intersect the rectangle 
$\,R=[-a/2,a/2\hh]\times[-b/2,b/2\hh]$. For $G=G_1,G_2,G_3,G_4,G_5$ we see this by applying the reasoning in the preceding paragraph to $n=2$, $e_1=(a,0)$ and $e_2=(0,b)$. For $G=G_6,G_7$, we note instead that $T_2^k(x,y)=(x,y+kb)$ and $T_4^l(x,y)=(x+la/2,y+lb/2)$ for all $k,l\in\mathds Z$. A suitable choice of $l$ will thus yield
the condition $x\in[0,a/2]$, and then some $k$ will ensure that $y\in[-b/2,b/2]$.\smallskip

Let the phrase ``we use $T$ to eliminate $S$'' now mean that a transformation $T\in G$ sends a set $S\subseteq R$ onto a subset of $R\smallsetminus S$ (and so the stated property still holds when $R$ is replaced by $R\smallsetminus S$).\smallskip

First, when $G=G_1, G_2,G_4$, $G_5$, $G_6$ and $G_7$ (or, when $G=G_3$), $T_3\w$ (or, $T_4\w$) 
may be used to eliminate $\,[-a/2,0)\times[-b/2,b/2\hh]$. Similarly, when $G=G_2$ (or, $G=G_3$ and $G=G_6$) the set $\,(a/4,a/2)\times[-b/2,b/2\hh]$ (or, 
$\,(0,a/2\hh]\times[-b/2,0)$) is eliminated with the aid of 
$T_4\w\circ T_3\w$ (or, respectively, $T_3\w$). Next, when $G=G_4$ or $G=G_7$, we eliminate $\,(a/4,a/2)\times[\hs0,b/2\hh]\,$ (or, 
$\,(a/4,a/2)\times[-b/2,0\hs]$) using $T_4\w\circ T_3\w$ (or, 
respectively, $T_2^{-\nnh1}\circ T_4\w\circ T_3\w$). Finally, in both 
cases $G=G_1$ or $G=G_5$, we use $T_2^{-\nnh1}\circ T_4\w$ (or, 
$T_4\w$) to eliminate $\,(a/4,a/2)\times[\hs0,b/2\hh]\,$ or, 
respectively, $\,(a/4,a/2)\times[-b/2,0\hs]$.
\end{proof}

\begin{lemma}\label{psixy} Every 
$T\in G\,$ is given by:
\begin{equation}\label{psi}
T(x,y)\,=\,(\delta x+ka/2,\,\ve y+lb/2)\hskip12pt\text{where}
\hskip7pt\delta,\ve\in\{1,\nnh-\nnh1\}\hskip6pt\text{and}\hskip6pt
k,l\in\mathds Z\hh,
\end{equation}
for all\/ $\,(x,y)\in\R^2\nh$. Furthermore, whenever\/ $T\in G\,$ is of 
the form\/ {\rm(\ref{psi})}, we have
\begin{equation}\label{sgp}
\begin{array}{rl}
\mathrm{(i)}\hskip18pt\text{when}\ G=G_1:&\hskip18pt(-\nnh1)^k\delta\ve\,
=\,(-\nnh1)^{k+l}=\,1\hh,\\
\mathrm{(ii)}\hskip18pt\text{when}\ G=G_2:&\hskip18pt(-\nnh1)^k\delta\ve\,=\,(-\nnh1)^l=\,1\hh,\\
\mathrm{(iii)}\hskip18pt\text{when}\ G=G_3:&\hskip18pt\ve\,=\,(-\nnh1)^k=\,1\hh,\\
\mathrm{(iv)}\hskip18pt\text{when}\ G=G_4:&\hskip18pt\delta\ve\,=\,(-\nnh1)^{k+l}=\,1\hh,\\
\mathrm{(v)}\hskip5pt\text{when}\ G=G_5, G_6:&\hskip18pt\ve\,=\,(-\nnh1)^{k+l}=\,1\hh,\\
\mathrm{(vi)}\hskip18pt\text{when}\ G=G_7:&\hskip18pt\ve\,=\,(-\nnh1)^{k+l}\delta=\,1\hh.
\end{array}
\end{equation}
\end{lemma}
\begin{proof}Af\-fine transformations of type (\ref{psi}) form a group, and 
each of the mappings assigning to $T$ with (\ref{psi}) the value 
$\,\delta$, or $\,\ve$, or $\,(-\nnh1)^k\nh$, or $\,(-\nnh1)^l\nh$, is a group 
homo\-mor\-phism into $\,\{1,\nnh-\nnh1\}$. This is immediate since, 
expressing (\ref{psi}) as $T\sim(\delta,\ve,k,l)$, we have 
$T\circ\hat T
\sim(\delta\hat\delta,\ve\hat\ve,k+\hat k,l+\ve \hat l\hs)\,$ 
and $T^{-\nnh1}\nh\sim(\delta,\ve,-\delta k,-\ve l)\,$ whenever 
$T\sim(\delta,\ve,k,l)\,$ and 
$\,\hat T\sim(\hat\delta,\hat\ve,\hat k,\hat l)$, while 
$\,\mathrm{Id}\sim(1,1,0,0)$. Each of the terms equated to $\,1\,$ in 
(\ref{sgp}) thus represents a group homo\-mor\-phism, and $\,G\,$ is, in each 
case, contained in its kernel, since so is, clearly, the generator set for 
each $\,G$, specified above.
\end{proof}
We can now describe the quotient space $\R^2\nnh/G$ of the action of 
$G$.
\begin{proposition}\label{quott}
If $G=G_3$ or $G=G_6$, two points of $\,Q=[\hs0,a/2\hh]\times[\hs0,b/2\hh]\,$ lie in the same\/ $\,G$-or\-bit if and only if they are $(x,0)\,$ and $\,(x,b/2)$ or, 
respectively, $(x,0)\,$ and $\,(-x+a/2,b/2)$, for some $\,x\in[\hs0,a/2\hh]$.

In the other five examples, a two-el\-e\-ment subset of\/ 
$\,Q=[0,a/4]\times[-b/2,b/2]\,$ is contained in a single $G$-orbit if and only if it is $\,\{(x,-b/2),(x,b/2)\}$, where $x\in[\hs0,a/4\hh]$, or is of the form:
\begin{itemize}
\item[(a)] for $G=G_1, G_2, G_4$ only:
$\,\,\{(0,y),(0,-y)\}\,\,$ for some\/ $\,y\in(0,b/2\hh]$.
\item[(b)] for $G=G_4$ only: $\,\,\{(a/4,y),(a/4,-y+b/2)\}\,\,$ for 
some\/ $\,y\in[\hs0,b/4)$.
\item[(c)] for $G=G_4$ only: $\,\,\{(a/4,y),(a/4,-y-b/2)\}\,\,$ for 
some\/ $\,y\in(-b/4,0\hs]$.
\item[(d)] for $G=G_1,G_5, G_7$ only: 
$\,\,\{(a/4,y),(a/4,y+b/2)\}\,\,$ for some\/ $\,y\in[-b/2,0\hs]$;
\item[(e)] for $G=G_7$ only: $\,\,\{(0,y),(0,y+b/2)\}\,\,$ for some\/ $\,y\in[-b/2,0\hs]$.
\end{itemize}
\end{proposition}
\begin{proof}The `if' part of our assertion follows if one uses: 
$T_4\w\circ T_3\w$ and $T_2\w$ in the first or, respectively, 
second paragraph of the lemma; $T_3\w$ for (a); $T_4\w\circ T_3\w$ 
and $T_2^{-\nnh1}\circ T_4\w\circ T_3\w$ in (b)--(c); 
$T_4\w$ and $T_2^{-\nnh1}\circ T_4\w$ for (d).

Let $\,Q\,$ now contain both $\,(x,y)\,$ and 
$\,(\hat x,\hat y)=T(x,y)\ne(x,y)$. Note that 
$\,(\hat x,\hat y)=(\delta x+ka/2,\,\ve y+lb/2)\,$ by (\ref{psi}). We will 
repeatedly invoke the obvious fact that
\begin{equation}\label{end}
\begin{array}{l}
|p-q|\,\le\,c\,\,\mathrm{\ whenever\ }\,\,p\,\,\mathrm{\ and\ 
}\,\,q\,\,\mathrm{\ lie\ in\ a\ closed\ interval\ of}\\
\mathrm{length\ }\,c\mathrm{,\ with\ equality\ only\ if\ }\,p\,\mathrm{\ and\ 
}\,q\,\mathrm{\ are\ the\ endpoints.}
\end{array}
\end{equation}
For $G=G_3,G_6$, $\,Q=[\hs0,a/2\hh]\times[\hs0,b/2\hh]$. 
Thus, with $\,\ve=1$, cf.\ (\ref{sgp}),
\begin{equation}\label{lie}
\mathrm{i)}\hskip6pt
\{2x/a\hh,\,\,k+2\delta x/a\}\,\subseteq\,[\hs0,1\hh]\hh,
\hskip22pt\mathrm{ii)}\hskip6pt\{2y/b\hh,\,\,l+2y/b\}\,
\subseteq\,[\hs0,1\hh]\hh. 
\end{equation}
As an immediate consequence of (\ref{end}), if $\,a,b\in(0,\infty)$,
\begin{equation}\label{psb}
\begin{array}{l}
\mathrm{for\ any\ }\,(\delta,k,x)\in\{1,-\nnh1\}\times\mathds Z\times[\hs0,a/2\hh]\,
\mathrm{\ with\ (\ref{lie}.i)\ one\ has\ }\,k\,=\,x\,=\,0\\
\mathrm{or\ }\hs(\delta,k)\in\{(1,0),(-\nnh1,1)\}\hs\mathrm{\ or\ 
}\hs(\delta,k,x)\in\{(1,1,0),(1,-\nnh1,a/2),(-\nnh1,2,a/2)\}\hh,\\
\mathrm{while\ }\,l=0\,\mathrm{\ or\ }\hs(y,l)=(0,1)\hs\mathrm{\ whenever\ 
}\,(l,y)\in\mathds Z\times[\hs0,b/2\hh]\,\mathrm{\ satisfy\ (\ref{lie}.ii).}
\end{array}
\end{equation}
Both when $G=G_3$ (where $\,k\,$ is even by (\ref{sgp}.ii)), and when $G=G_6$
(with evenness of $\,k$ as an additional assumption), (\ref{psb}) gives 
$\,k=0\,$ or $\,(\delta,k,x)=(-\nnh1,2,a/2)$. Hence $\,\hat x=x$ (as 
$\,\delta=1\,$ when $\,k=0$, unless $\,x=0$, cf.\ (\ref{psb})). The resulting 
relation $\,y\ne\hat y=y+lb/2$ means that $\,l\ne0$, and so, from (\ref{psb}), 
$\,(y,l)=(0,1)$. This proves our claim about $G_3$, and at the same time 
contradicts (\ref{sgp}.iv) for $G_6$, so that, when $G=G_6$, $\,k\,$ and 
$\,l\,$ must be odd. Now (\ref{psb}) yields the required assertion about 
$G_6$.

For $G=G_1,G_2,G_4,G_5,G_7$, our assumption that 
$\,Q=[\hs0,a/4\hh]\times[-b/2,b/2\hh]$ contains both $\,(x,y)\,$ and 
$\,(\hat x,\hat y)\ne(x,y)=(\delta x+ka/2,\,\ve y+lb/2)\,$ has an 
immediate consequence:
\begin{equation}\label{thr}
\hat x=x\hh,\hskip7pt\mathrm{and\ }\,(\delta,k)=(1,0)\,\mathrm{\ or\ }\,
(\delta,k,x)=(-\nnh1,0,0)\,\mathrm{\ or\ }\,(\delta,k,x)=(-\nnh1,1,a/4)\hh.
\end{equation}
(Namely, both $\,4x/a\,$ and $\,2k+4\delta x/a\,$ lie in $\,[\hs0,1\hh]$, and 
so, due to (\ref{end}), $\,k=0\,$ if $\,\delta=1$, while the case 
$\,\delta=-\nnh1$ allows just two possibilities: $\,(k,x)=(0,0)\,$ or 
$\,(k,x)=(1,a/4)$).) The relation $\,(x,y)\ne(\hat x,\hat y)\,$ must thus be 
realized by the $\,y\,$ components, that 
is, $\,y\,$ and $\,\ve y+lb/2\,$ are two different numbers in 
$\,[-b/2,b/2\hh]\,$ or, equivalently,
\begin{equation}\label{eqv}
2y/b\,\mathrm{\ is\ different\ from\ }\,l+2\ve y/b\,\mathrm{\ and\ both\ lie\ 
in\ }\,[-\nnh1,1\hh]\mathrm{,\ while\ }-\nnh2\le l\le2\hh,
\end{equation}
the last conclusion being obvious from (\ref{end}).

First, if $\,|l|=2$, (\ref{end}) leads to the case mentioned in 
the second paragraph of the lemma. We may therefore assume from now on that 
$\,|l|\le1$.

Next, suppose that $\,l=0$, and so $\,\ve=-\nnh1\,$ by (\ref{eqv}). 
By (\ref{sgp}.iv), this excludes the cases $G=G_5$ and $G=G_6$. In the 
remaining cases $G=G_1,G_2,G_4$, (\ref{sgp}) gives 
$\,(-\nnh1)^k\delta=-\nnh1\,$ which, combined with (\ref{thr}) and 
(\ref{psi}), shows that $\,\hat x=x=0\,$ and $\,\hat y=-y$. Hence (a) follows.

Finally, let $\,|l|=1$, so that (\ref{sgp}) excludes the case $G=G_2$, 
yields $\,(-\nnh1)^k\nh=-\nnh1\,$ for $G=G_1,G_4,G_5$, and 
$\,(-\nnh1)^k\delta=-\nnh1\,$ for $G=G_7$.
Thus, by (\ref{thr}), $\,(\delta,k,x)=(-\nnh1,1,a/4)\,$ or, for $G=G_7$ only, $\,(\delta,k,x)=(-\nnh1,0,0)$, while (\ref{sgp}) provides a 
specific value of $\,\ve\,$ in each case. Using (\ref{thr}) and (\ref{psi}) 
again, we obtain (b), (c), (d) or (e).
\end{proof}

From Proposition~\ref{quott}, one easily arrives to the following conclusions:
\begin{corollary}\label{thm:corconclusions}
The quotient $\R^2/G$ is identified with the following flat $2$-orbifold (cf.~Table~\ref{table:flat2orbifolds}):
\begin{itemize}
\item[(a)] $\R P^2(2,2;)$, if $G=G_1$ (see Figures~\ref{fig:caseG1-1} and \ref{fig:caseG1-2});
\item[(b)] $D^2(2,2;)$, if $G=G_2$ (see Figures~\ref{fig:caseG2-1}---\ref{fig:caseG2-4});
\item[(c)] $S^1\times I$ if $G=G_3$ (see Figure~\ref{fig:caseG3-1});
\item[(d)] $S^2(2,2,2,2;)$ if $G=G_4$ (see Figures~\ref{fig:caseG4-1}---\ref{fig:caseG4-3});
\item[(e)] M\"obius band if $G=G_5$ and when $G=G_6$ (see Figures~\ref{fig:caseG5-1}---\ref{fig:caseG5-2} and \ref{fig:caseG6-1}---\ref{fig:caseG6-2});
\item[(f)] Klein bottle if $G=G_7$ (see Figures~\ref{fig:caseG7-1}---\ref{fig:caseG7-3}).
\end{itemize}
\end{corollary}

\section{Alternative proof of Proposition~\ref{prop:toritotori}}\label{sec:appendixB}

We now provide an alternative and elementary proof of Proposition~\ref{prop:toritotori}, without making use of results on equivariant Gromov-Hausdorff convergence.

In the noncollapsing case $\Vol(T^n,\g_i)\ge v_0>0$, the dimension of the limit space is $n$, and by the classical compactness theorem for lattices of Mahler~\cite{Mahler46}, the limit is again a flat torus $(T^n,\h)$. In the collapsing case, we have the following:

\begin{claim}\label{claim:sequenceofbases}
Suppose that $\Vol(T^n,\g_k)\searrow 0$ as $k\nearrow+\infty$. For each $k\in\N$, choose a lattice $L^{(k)}\subset\R^n$ such that $\g_k$ is isometric to the induced metric $\g_{L^{(k)}}$ on $\R^n/L^{(k)}$.
Up to passing to subsequences, there exists a $\mathds Z$-basis 
$\big\{v_1^{(k)},\ldots,v_n^{(k)}\big\}$ of $L^{(k)}$ such that for some $0\leq m<n$:
\begin{enumerate}[\rm (i)]
\item If $m\geq1$, then $w_j:=\lim\limits_{k\to\infty} v_j^{(k)}$ are linearly independent for $j=1,\dots, m$;
\item $\lim\limits_{k\to\infty} v_j^{(k)}=0$ for $j=m+1,\dots, n$.
\end{enumerate}
\end{claim}

Note that Claim~\ref{claim:sequenceofbases} concludes the proof of Proposition~\ref{prop:toritotori}. 
Indeed, set
\begin{align*}
V^{(k)}&=\operatorname{span}_\R\big\{v_{1}^{(k)},\ldots, v_m^{(k)}\big\}, &\Lambda^{(k)}&=\operatorname{span}_\Z\big\{v_{1}^{(k)},\ldots, v_m^{(k)}\big\},\\
V&=\operatorname{span}_\R\big\{w_{1},\ldots, w_m\big\}, &\Lambda&=\operatorname{span}_\Z\big\{w_{1},\ldots, w_m\big\}.
\end{align*}
For all $k$, the flat torus $(T^{m},\h_k)$ given by the quotient $V^{(k)}/\Lambda^{(k)}$ is a closed subspace of $(T^n,\g_k)$. By property (ii), there are $r_k\searrow 0$, such that the $r_k$-neighborhoods of $(T^{m},\h_k)$ in $(T^n,\g_k)$ are the entire space $(T^n,\g_k)$. 
Moreover, the Gromov-Hausdorff limit of $(T^{m},\h_k)$ is clearly the flat torus $(T^{m},\h)$ given by $V/\Lambda$. Thus, $(T^{m},\h)$ is the Gromov-Hausdorff limit of $(T^n,\g_k)$.

In order to prove Claim~\ref{claim:sequenceofbases}, given a lattice $\Lambda\subset\R^n$, we want to determine a $\mathds Z$-basis of $\Lambda$ in which:
\begin{itemize}
\item the angle between any vector of the basis and the hyperplane generated by the others is greater than or equal to some positive constant depending only on the dimension $n$;\smallskip

\item the length of each vector of the basis is bounded above by some multiple of the diameter of the flat torus $\R^n/\Lambda$.
\end{itemize}
There are related notions in the literature  (``short bases'', ``reduced bases'', etc.);
we use here a construction based on the notion of \emph{$\lambda$-normal basis} introduced  in \cite{BuserKarcher81}.
The results in this section are \emph{far} from being optimal, but they serve for our purposes.

Let us denote by $\langle\cdot,\cdot\rangle$ the Euclidean inner product in $\R^n$, and let us recall that the angle between two nonzero vectors $v,w\in\R^n$ is the number \[\widehat{vw}=\arccos\left(\frac{\langle v,w\rangle}{\vert v\vert\,\vert w\vert}\right)\in[0,\pi].\]
Given a nonzero vector $v\in\R^n$ and a nontrivial subspace $W\subset\R^n$, the angle between $v$ and $W$ is:
\[\mathrm{ang}(v,W)=\min\big\{\widehat{vw}:w\in W\setminus\{0\}\big\}\in\left[0,\tfrac\pi2\right];\]
clearly, if $W'\subset W$, then $\mathrm{ang}(v,W')\ge\mathrm{ang}(v,W)$.

\begin{lemma}\label{thm:lowboundangle}
There exists a constant $\theta_n\in\left]0,\frac\pi2\right]$ that depends only on the dimension $n\ge2$ such that every lattice $\Lambda\subset\R^n$ admits a $\mathds Z$-basis 
$v_1,\ldots,v_n$ such that, denoting by $W_i\subset\R^n$ the subspace spanned by $v_1,\ldots,\hat v_i,\ldots,v_n$, one has:
\begin{equation}\label{eq:lowboundangle}
\mathrm{ang}(v_i,W_i)\ge\theta_n.
\end{equation}
\end{lemma}
\begin{proof}
Given $\Lambda$,
it is proven in \cite[4.1.3~Proposition and 4.1.4~Proposition, p.~50--51]{BuserKarcher81} the existence of a $\mathds Z$-basis $v_1,\ldots,v_n$ of
$\Lambda$ such that:
\[\det(v_1,\ldots,v_n)\ge2^{-\frac14n(n-1)}\prod_{i=1}^n \vert v_i\vert.\]
Such a basis satisfies \eqref{eq:lowboundangle} for $\theta_n=\arcsin\big(2^{-\frac14n(n-1)}\big)$.
\end{proof}
Let now $R_0(\Lambda)>0$ be defined by:
\[
R_0(\Lambda)=\min\big\{r>0:\ \overline B(r)\ \text{contains a $\mathds Z$-basis $v_1,\ldots,v_n$ of $\Lambda$ satisfying \eqref{eq:lowboundangle}}\big\},
\]
where $\overline B(r)\subset\R^n$ denotes the closed ball of radius $r$ centered at $0$.

The set of $\mathds Z$-bases $w_1,\ldots,w_n$ of $\Lambda$ satisfying \eqref{eq:lowboundangle} and contained in $\overline B\big(R_0(\Lambda)\big)$ is finite.\footnote{%
Because the set of vectors $w\in\Lambda$ with $\vert w\vert\le R_0$ is finite.} We can rearrange the elements of each such basis in such a way that:
\begin{equation}\label{eq:order}
\vert w_1\vert\ge\vert w_2\vert\ge\ldots\ge\vert w_n\vert.
\end{equation}
We give the lexicographical order to the set of all ordered $\mathds Z$-basis $(w_1,\ldots,w_n)$ contained in $\overline B\big(R_0(\Lambda)\big)$ satisfying \eqref{eq:lowboundangle} and \eqref{eq:order}, and we denote by $(u_1,\ldots,u_n)$ a \emph{minimal} element of this set. We call such a minimal element a \emph{special basis of $\Lambda$}.
Clearly, if $(u_1,\ldots,u_n)$ is a special basis of $\Lambda$, then $\vert u_1\vert=R_0(\Lambda)$. 
\smallskip

Let us also recall that the diameter of the flat torus $\R^n/\Lambda$ is given by:
\begin{equation}\label{eq:diameter}
\mathrm{diam}(\R^n/\Lambda)=\max\big\{d>0:\exists\ v\in\R^n\ \text{with}\ \vert v\vert=d\ \text{and}\ B(v;d)\cap\Lambda=\emptyset\big\},
\end{equation}
where $B(v;d)$ denotes the open ball centered at $v$ and of radius $d$.
\begin{proposition}\label{thm:bounddiameter}
There exists a positive constant $\beta_n$ depending only on $n$ such that, given a lattice $\Lambda\subset\R^n$ and a special basis $(u_1,\ldots,u_n)$ of $\Lambda$, 
the diameter of the flat torus $\R^n/\Lambda$ is greater than or equal to $\beta_n \,\vert u_1\vert$.
\end{proposition}

\begin{proof}
Using \eqref{eq:diameter}, the result will follow if we show the existence of $\beta_n>0$ such that
\begin{equation}\label{eq:ineqdiam}
\vert \beta_n\, u_1-\epsilon\,u_1-\gamma\vert^2\ge \beta_n^2\,\vert u_1\vert^2
\end{equation}
for all $\epsilon\in\mathds Z$ and all $\gamma\in\mathrm{span}_\mathds Z\{u_2,\ldots,u_n\}$.

In order to prove the desired inequality, we start observing that, by minimality, given $\gamma\in\mathrm{span}_\mathds Z\{u_2,\ldots,u_n\}$, we have:\footnote{%
From \eqref{eq:fromminimality}  it also follows that $\vert\gamma\vert^2\ge2\vert\langle u_1,\gamma\rangle\vert$, for all $\gamma\in\mathrm{span}_\mathds Z\{u_2,\ldots,u_n\},$ however this inequality is not used here.}
\begin{equation}\label{eq:fromminimality}
\vert u_1+\gamma\vert^2\ge\vert u_1\vert^2.
\end{equation}
Namely, if $\vert u_1+\gamma\vert<\vert u_1\vert$, then $\{w_1:=u_1+\gamma,u_2,\ldots,u_n\}$ would be another $\mathds Z$-basis of $\Lambda$ satisfying 
\eqref{eq:lowboundangle}  and contained in $\overline B\big(R_0(\Lambda)\big)$. The elements of such a basis could then be rearranged to obtain an ordered basis
which is strictly less than $(u_1,\ldots,u_n)$ in the lexicographical order. But this contradicts the minimality of $(u_1,\ldots,u_n)$  and shows that \eqref{eq:fromminimality} must hold.

Note that \eqref{eq:fromminimality} implies that, when $\epsilon=0$, \eqref{eq:ineqdiam} holds for all $\beta_n\in[0,1]$ and all $\gamma\in\mathrm{span}_\mathds Z\{u_2,\ldots,u_n\}$.
On the other hand, by Lemma~\ref{thm:lowboundangle}:
\begin{equation}\label{eq:anggammav1}
\phantom{\quad\forall\,\gamma\in\mathrm{span}_\mathds Z\{u_2,\ldots,u_n\}}
\mathrm{ang}(\gamma,u_1)\ge\theta_n,\quad\forall\,\gamma\in\mathrm{span}_\mathds Z\{u_2,\ldots,u_n\}.
\end{equation}
Let us set 
\begin{equation}\label{eq:defbetan}
\beta_n=\min\big\{\tfrac12,\ \sin(2\theta_n)\big\}
\end{equation}
and show that, with such a choice\footnote{%
$\beta_2=\beta_3=\frac12$, $\beta_n=\sin(2\theta_n)$ for all $n\ge4$.}, inequality \eqref{eq:ineqdiam} holds. Towards this goal, for all $\gamma\in\mathrm{span}_\mathds Z\{u_2,\ldots,u_n\}$, denote by $\gamma_1=\frac{\langle\gamma,u_1\rangle}{\langle u_1,u_1\rangle}u_1$ the orthogonal projection of $\gamma$ in the direction of $u_1$, and by
$\gamma_\perp=\gamma-\gamma_1$ the component of $\gamma$ orthogonal to $u_1$. 

Note that if $\gamma\in\mathrm{span}_\mathds Z\{u_2,\ldots,u_n\}$ is such that $\vert\gamma_\perp\vert\ge\beta_n\,\vert u_1\vert$, then clearly, for all $\epsilon\in\mathds Z$, 
the vector $\gamma+\epsilon\, u_1$ cannot belong to the open ball centered at $\beta_n\, u_1$ and of radius $\beta_n\,\vert u_1\vert$. Thus, in order to check \eqref{eq:ineqdiam}, it suffices to consider those $\gamma\in\mathrm{span}_\mathds Z\{u_2,\ldots,u_n\}$ satisfying:
\begin{equation}\label{eq:normgammaperp}
\vert\gamma_\perp\vert<\beta_n\,\vert u_1\vert.
\end{equation}
Moreover, such a $\gamma$ must satisfy \eqref{eq:fromminimality}, which implies that:
\begin{equation}\label{eq:ineqgamma1}
\sqrt{\beta_n^2\,\vert u_1\vert^2-\vert\gamma_\perp\vert^2}-\beta_n\,\vert u_1\vert\le\vert\gamma_1\vert\le \beta_n\,\vert u_1\vert-\sqrt{\beta_n^2\,\vert u_1\vert^2-\vert\gamma_\perp\vert^2}.
\end{equation}
For such a $\gamma$, it is easy to see that $\big\vert\gamma+(\epsilon-\beta_n)\, u_1\big\vert>\beta_n\vert u_1\vert$ for all $\epsilon\in\mathds Z$, see Figure~\ref{fig:gamma}.
\begin{figure}
\centering
\includegraphics[scale=0.5]{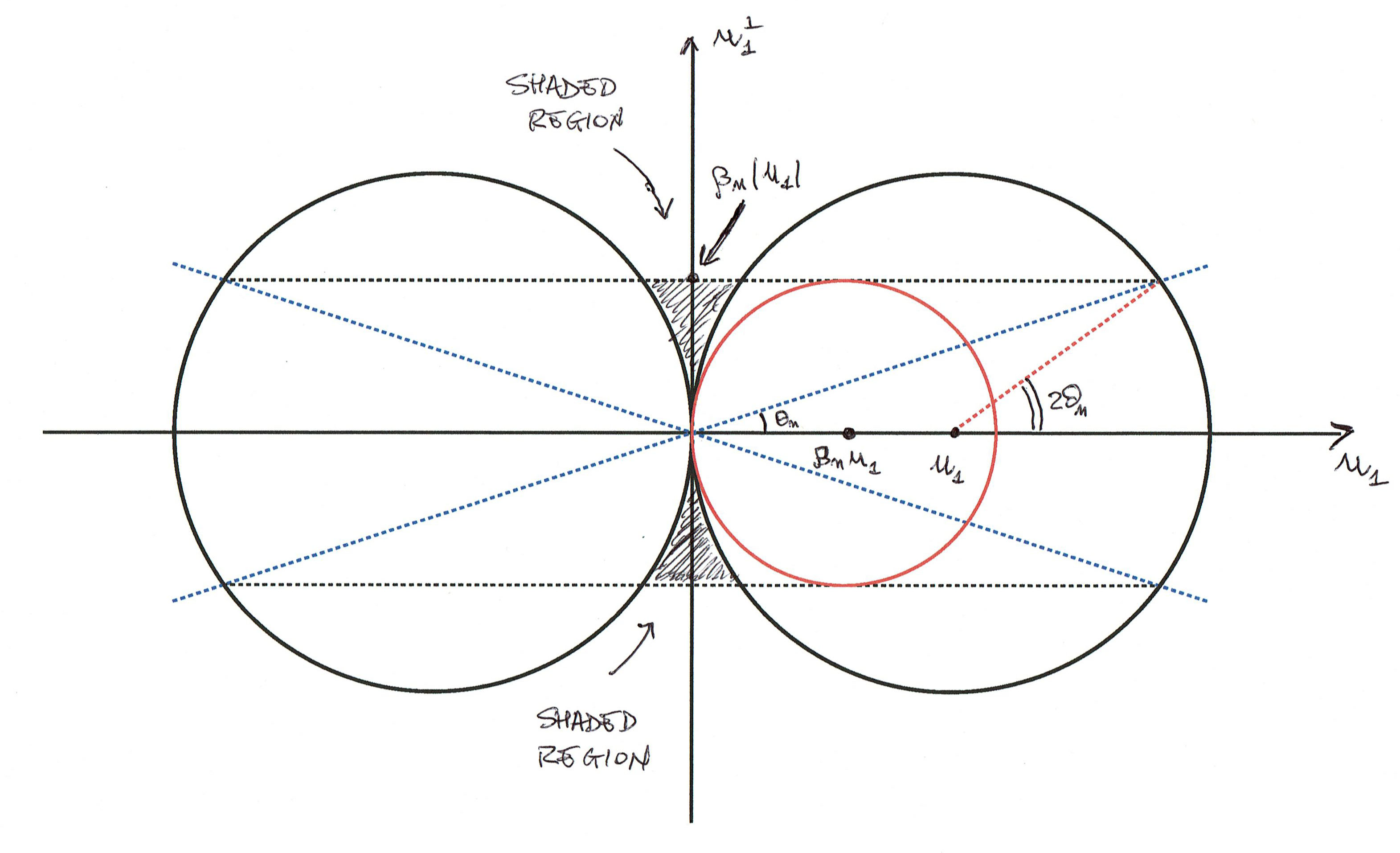}
\caption{An element $\gamma\in\mathrm{span}_\mathds Z\{u_2,\ldots,u_n\}$ satisfying \eqref{eq:fromminimality}, \eqref{eq:anggammav1} and \eqref{eq:normgammaperp} must lie in the shaded region of the picture.
Gven such a $\gamma$, any vector of the form $\gamma+\epsilon\, u_1$, with $\epsilon\in\mathds Z$, does not belong to the ball centered at $\beta_n\, u_1$ and of radius $\beta_n\,\vert u_1\vert$.}
\label{fig:gamma}
\end{figure}
Namely, for $\epsilon\ge1$, since $\beta_n\le\frac12$:
\begin{equation*}
2\sqrt{\beta_n^2\,\vert u_1\vert^2-\vert\gamma_\perp\vert^2}-2\beta_n\,\vert u_1\vert+\epsilon\,\vert u_1\vert \ge2\sqrt{\beta_n^2\,\vert u_1\vert^2-\vert\gamma_\perp\vert^2}+1(1-2\beta_n)\,\vert u_1\vert>0,
\end{equation*}
and this shows that  $\gamma+\epsilon\, u_1$ does not belong to the ball centered at $\beta_n\, u_1$ of radius $\beta_n\,\vert u_1\vert$. For $\epsilon\le0$, the vector $\gamma+\epsilon\, u_1$ does not even belong to the ball centered at $u_1$ with radius $\vert u_1\vert$,  completing the proof.
\end{proof}

\begin{proof}[Proof of Claim~\ref{claim:sequenceofbases}]
It suffices to choose $(v_1^{(k)},\ldots,v_n^{(k)})$ a special basis (ordered by decreasing norm of its elements) of $L^{(k)}$, as in Proposition~\ref{thm:bounddiameter}.
Since the diameter of $\R^n/L^{(k)}$ is bounded, then $\vert v_j^{(k)}\vert$ is a bounded sequence for all $j=1,\ldots,n$. Thus, up to taking subsequences, we can assume that each $v_j^{(k)}$ is convergent to some $w_j\in\R^n$.  There exists $j_0\in\{1,\ldots,n\}$ such that $w_j=0$ for $j> j_0$, and $w_j\ne0$ for $j\le j_0$. We claim that $w_1,\ldots,w_{j_0}$ is a linearly independent set. Namely, if $w_s$ were a linear combination of
$w_1,\ldots,\widehat w_s,\ldots,w_{j_0}$ for some $s\in\{1,\ldots,j_0\}$, then the angle between $v_s^{(k)}$ and the space generated by $v_1^{(k)},\ldots,\widehat {v_s}^{(k)},\ldots,v_{j_0}^{(k)}$ would be arbitrarily small as $k\to\infty$.
But for a special basis this is not possible, and this contradiction shows that $w_1,\ldots,w_{j_0}$ is a linearly independent set.
\end{proof}

\end{document}